\documentclass[11pt, letterpaper]{article}
\usepackage[leqno]{amsmath}
\usepackage{abstract, amsthm, etoolbox, enumitem, hyperref, mathtools, sectsty, thmtools, tocloft}
\usepackage[all]{xy}
\pdfoutput=1

\usepackage[T1]{fontenc}
\usepackage[frenchstyle, largesmallcaps, light]{kpfonts}
\usepackage[mathcal]{eucal}
\usepackage{microtype}
\usepackage{geometry}
\allsectionsfont{\centering\fontseries{m}\scshape}
\sectionfont{\raggedright\fontseries{m}\scshape}

\allowdisplaybreaks

\newlength{\bibitemsep}
\newlength{\bibparskip}
\let\oldthebibliography\thebibliography
\renewcommand\thebibliography[1]{%
  \oldthebibliography{#1}%
  \setlength{\parskip}{\bibitemsep}%
  \setlength{\itemsep}{\bibparskip}%
}

\apptocmd{\thebibliography}{\raggedright}{}{}

\newcommand{\bs}{\boldsymbol}
\newcommand{\mb}{\mathbf}
\newcommand{\mc}{\mathcal}
\newcommand{\mf}{\mathfrak}
\newcommand{\mr}{\mathrm}

\declaretheoremstyle[spaceabove=6pt, spacebelo=6pt, headfont=\scshape, bodyfont=\itshape, postheadspace=.5em]{thm}
\declaretheoremstyle[spaceabove=6pt, spacebelo=6pt, headfont=\scshape, bodyfont=\normalfont, postheadspace=.5em]{defn}

\theoremstyle{thm}

\newtheorem{thm}{Theorem}
\numberwithin{thm}{section}
\newtheorem{cor}[thm]{Corollary}
\newtheorem{lemma}[thm]{Lemma}
\newtheorem{prop}[thm]{Proposition}
\newtheorem*{thmA}{Theorem \ref{hodge.7}}
\newtheorem*{thmB}{Theorem \ref{rect.6}}
\newtheorem*{corC}{Corollary \ref{int.9}} 

\theoremstyle{defn}

\newtheorem{defn}[thm]{Definition} 
\newtheorem{ex}[thm]{Example}
\newtheorem*{motivation}{Motivation}
\newtheorem{notation}[thm]{Notation}
\newtheorem{rmk}[thm]{Remark}
\newtheorem*{summary}{Summary}
\numberwithin{equation}{thm}

\setlist[enumerate,1]{label=\textup{(\textit{\roman*})},itemsep=.125\baselineskip,parsep=0pt,topsep=.125\baselineskip,listparindent=\parindent,itemindent=\parindent,wide=\parindent}
\setlist[enumerate,2]{label=\textup{(\alph*)},itemsep=.125\baselineskip,parsep=0pt,topsep=.125\baselineskip,listparindent=\parindent,leftmargin=3\parindent}
\setlist[enumerate,3]{itemsep=.125\baselineskip,parsep=0pt,topsep=.125\baselineskip,listparindent=\parindent,leftmargin=\parindent}
\setlist[itemize]{itemsep=.125\baselineskip,parsep=0pt,topsep=.125\baselineskip,listparindent=\parindent}

\usepackage{array, longtable}

\DeclarePairedDelimiter{\delimpar}{(}{)}
\DeclarePairedDelimiter{\delimbrk}{[}{]}
\DeclarePairedDelimiter{\delimbrc}{\{}{\}}
\DeclarePairedDelimiter{\delimbar}{\lvert}{\rvert}
\DeclarePairedDelimiter{\delimang}{\langle}{\rangle}

\newcommand{\prns}[2][0]{%
  \ifcase#1\relax
    \delimpar{#2}\or
    \delimpar[\big]{#2}\or
    \delimpar[\Big]{#2}\or
    \delimpar[\bigg]{#2}\or
    \delimpar[\Bigg]{#2}
  \else
    \delimpar*{#2}
  \fi}
\newcommand{\brk}[2][0]{%
  \ifcase#1\relax
    \delimbrk{#2}\or
    \delimbrk[\big]{#2}\or
    \delimbrk[\Big]{#2}\or
    \delimbrk[\bigg]{#2}\or
    \delimbrk[\Bigg]{#2}
  \else
    \delimbrk*{#2}
  \fi}
\newcommand{\brc}[2][0]{%
  \ifcase#1\relax
    \delimbrc{#2}\or
    \delimbrc[\big]{#2}\or
    \delimbrc[\Big]{#2}\or
    \delimbrc[\bigg]{#2}\or
    \delimbrc[\Bigg]{#2}
  \else
    \delimbrc*{#2}
  \fi}
\newcommand{\bars}[2][0]{%
  \ifcase#1\relax
    \delimbar{#2}\or
    \delimbar[\big]{#2}\or
    \delimbar[\Big]{#2}\or
    \delimbar[\bigg]{#2}\or
    \delimbar[\Bigg]{#2}
  \else
    \delimbar*{#2}
  \fi}
\newcommand{\ang}[2][0]{%
  \ifcase#1\relax
    \delimang{#2}\or
    \delimang[\big]{#2}\or
    \delimang[\Big]{#2}\or
    \delimang[\bigg]{#2}\or
    \delimang[\Bigg]{#2}
  \else
    \delimang*{#2}
  \fi}

\newcommand{\scr}[3][0]{\ifblank{#2}{\mathstrut{\vphantom{\prns[#1]{}}}_{#3}}{\ifblank{#3}{\mathstrut{\vphantom{\prns[#1]{}}}^{#2}}{\mathstrut{\vphantom{\prns[#1]{}}}^{#2}_{#3}}}}

\makeatletter
\def\scripts#1#2#3{\def\scripts@{\prns[#1]{#3}}\def\scripts@@{#2}\def\scripts@@@{#2}\@ifnextchar^\@sup\@nsup}
\def\@sup^#1{\def\scripts@@@{\scripts@@^{#1}}\@ifnextchar_\@supsub{\scripts@@@\scripts@}}
\def\@supsub_#1{\scripts@@@_{#1}\scripts@}
\def\@nsup{\@ifnextchar_{\@sub}{\scripts@@@\scripts@}}
\def\@sub_#1{\def\scripts@@@{\scripts@@_{#1}}\@ifnextchar^\@subsup{\scripts@@@\scripts@}}
\def\@subsup^#1{\scripts@@@^{#1}\scripts@}
\makeatother

\renewcommand{\1}[1]{\mathbf{1}_{#1}}
\newcommand{\ab}{\operatorname{\mathcal{A}b}}
\newcommand{\alg}[2][0]{\scripts{#1}{\operatorname{\mathcal{A}lg}}{#2}}
\newcommand{\altdgnerve}[2][0]{\operatorname{\overline{\mathfrak{N}}}_{\mathrm{dg}}\ifblank{#2}{}{\prns[#1]{#2}}}

\newcommand{\calg}[2][0]{\scripts{#1}{\operatorname{\mathcal{CA}lg}}{#2}}
\newcommand{\cat}{\operatorname{\mathcal{C}at}}
\newcommand{\CAT}{\operatorname{\mathcal{CAT}}}
\newcommand{\colim}{\operatorname*{colim}}
\newcommand{\cpx}[2][0]{\scripts{#1}{\operatorname{\mathcal{C}px}}{#2}}

\renewcommand{\d}{\operatorname{d}}
\renewcommand{\D}[2][0]{\scripts{#1}{\operatorname{\mathcal{D}}}{#2}}
\newcommand{\tD}[2][0]{\scripts{#1}{\operatorname{\widetilde{\mathcal{D}}}}{#2}}
\newcommand{\DA}[2][0]{\scripts{#1}{\operatorname{\mathcal{DA}}}{#2}}

\newcommand{\DGCAT}{\operatorname{\mathcal{DGCAT}}}
\newcommand{\dgnerve}[2][0]{\operatorname{\mathfrak{N}}_{\mathrm{dg}}\ifblank{#2}{}{\prns[#1]{#2}}}
\renewcommand{\dim}[2][0]{\operatorname{dim}\ifblank{#2}{}{\prns[#1]{#2}}}
\newcommand{\ext}[3][0]{\ifblank{#2}{\operatorname{ext}}{\scripts{#1}{\operatorname{ext}}{#2,\, #3}}}
\newcommand{\fin}{\operatorname{\mathcal{F}in}_*}
\newcommand{\fun}[3][0]{\ifblank{#2}{\operatorname{\mathcal{F}un}}{\scripts{#1}{\operatorname{\mathcal{F}un}}{#2,\, #3}}}

\newcommand{\gr}[2][0]{\operatorname{gr}\ifblank{#2}{}{\scripts{#1}{}{#2}}}
\newcommand{\h}{\mathfrak{h}}
\renewcommand{\H}[2][0]{\scripts{#1}{\operatorname{\mathcal{H}}}{#2}}
\newcommand{\ho}[2][0]{\ifblank{#2}{\operatorname{ho}}{\scripts{#1}{\operatorname{ho}}{#2}}}

\renewcommand{\hom}[3][0]{\ifblank{#2}{\operatorname{hom}}{\scripts{#1}{\operatorname{hom}}{#2,\, #3}}}
\newcommand{\id}{\operatorname{id}}

\newcommand{\ind}[2][0]{\operatorname{\mathcal{I}nd}\ifblank{#2}{}{\prns[#1]{#2}}}
\newcommand{\inftyone}{\ensuremath{(\infty,1)}}
\newcommand{\isom}{\stackrel{\hspace{-.1em}\raisebox{0.1em}{\smash{\ensuremath{\sim}}}}{\smash{\ensuremath{\to}}}}
\newcommand{\kk}{\boldsymbol{\kappa}}
\newcommand{\K}[2][0]{\scripts{#1}{\operatorname{\mathcal{K}}}{#2}}
\newcommand{\KK}{\mathbf{K}}

\newcommand{\map}[3][0]{\ifblank{#2}{\operatorname{map}}{\scripts{#1}{\operatorname{map}}{#2,\, #3}}}
\newcommand{\mhc}{\operatorname{\mathcal{MHC}}}
\newcommand{\MHC}{\operatorname{\mathbf{MHC}}}
\newcommand{\mhm}[2][0]{\operatorname{\mathcal{MHM}}\scripts{#1}{}{#2}}
\newcommand{\mhs}{\operatorname{\mathcal{MHS}}}
\renewcommand{\mod}[2][0]{\ifblank{#2}{\operatorname{\mathcal{M}od}}{\scripts{#1}{\operatorname{\mathcal{M}od}}{#2}}}

\newcommand{\nerve}[2][0]{\operatorname{\mathfrak{N}}\ifblank{#2}{}{\prns[#1]{#2}}}
\newcommand{\on}{\operatorname}
\newcommand{\op}{^{\mathrm{op}}}
\newcommand{\pr}{\operatorname{\mathcal{P}r}}
\newcommand{\psh}[3][0]{\operatorname{\mathcal{PS}h}\ifblank{#2}{}{\prns[#1]{#2,\, #3}}}
\newcommand{\qcat}{\operatorname{\mathcal{QC}at}}
\newcommand{\QCAT}{\operatorname{\mathcal{QCAT}}}
\newcommand{\qcoh}[2][0]{\operatorname{\mathcal{QC}oh}\ifblank{#1}{}{\prns[#1]{#2}}}
\newcommand{\resp}[1]{\textup{(}resp.\@ #1\textup{)}}

\newcommand{\sch}[2][]{\mathrm{\mathcal{S}ch}^{\textup{#1}}\ifblank{#2}{}{_{/#2}}}
\newcommand{\set}{\operatorname{\mathcal{S}et}}
\newcommand{\sh}[3][0]{\ifblank{#2}{\operatorname{\mathcal{S}h}}{\scripts{#1}{\operatorname{\mathcal{S}h}}{#2,\, #3}}}
\newcommand{\SH}[2][0]{\operatorname{\mathcal{SH}}\ifblank{#2}{}{\scripts{#1}{}{#2}}}
\newcommand{\sm}[2][]{\mathrm{\mathcal{S}m}^{\textup{#1}}\ifblank{#2}{}{_{/#2}}}

\newcommand{\snerve}[2][0]{\operatorname{\mathfrak{N}_{\Delta}}\ifblank{#2}{}{\prns[#1]{#2}}}
\newcommand{\spc}[2][0]{\ifblank{#2}{\operatorname{\mathcal{S}pc}}{\scripts{#1}{\operatorname{\mathcal{S}pc}}{#2}}}
\newcommand{\spec}[2][0]{\operatorname{Spec}\prns[#1]{#2}}

\newcommand{\sset}{\operatorname{\mathcal{S}et}_{\Delta}}
\newcommand{\sym}[2][0]{\operatorname{sym}\ifblank{#2}{}{\prns[#1]{#2}}}
\newcommand{\trun}{\mathfrak{t}}

\newcommand{\umor}[3][0]{\ifblank{#2}{\mathbf{mor}}{\scripts{#1}{\mathbf{mor}}{#2,\, #3}}}
\newcommand{\yon}[3][0]{\operatorname{\mathbf{y}}_{#2}{\ifblank{#3}{}{\prns[#1]{#3}}}}

\title{\textsc{Rectification of Deligne's mixed Hodge structures}}
\author{Brad Drew}
\date{}

\begin{document}
\maketitle

\begin{abstract}
\noindent
We promote Beilinson's triangulated equivalence between the bounded derived category of rational polarizable mixed Hodge structures and the derived category of rational polarizable mixed Hodge complexes to an equivalence of symmetric monoidal quasi-categories.
We use this equivalence to construct a presheaf of commutative differential graded algebras in the ind-completion of the category of rational mixed Hodge structures which computes Deligne's mixed Hodge structure on the rational Betti cohomology of $\mb{C}$-schemes of finite type.
This leads to a presheaf\textemdash in the quasi-categorical sense\textemdash of $\mb{E}_{\infty}$-algebras computing integral mixed Hodge structures.
\end{abstract}

\tableofcontents

\subsection*{Introduction}

\noindent
Let $\sch[ft]{\mb{C}}$ denote the category of $\mb{C}$-schemes of finite type and let $X \in \sch[ft]{\mb{C}}$.
The derived category $\mr{D}\prns{X\prns{\mb{C}}^{\mr{an}}, \mb{Q}}$ of analytic sheaves of $\mb{Q}$-modules, M.~Saito's derived category $\mr{D}^{\mr{b}}\mhm{X}$ of mixed Hodge modules (\cite{Saito_modules-de-hodge, Saito_mixed-hodge}) and the $\mb{P}^1$-stable $\mb{A}^1$-homotopy category $\mc{SH}\prns{X}$ of F.~Morel and V.~Voevodsky, as developed by J.~Ayoub (\cite{Ayoub_six-operationsI, Ayoub_six-operationsII}), admit Grothendieck six-functor formalisms.
They are related by symmetric monoidal triangulated functors $\omega^*_X: \mr{D}^{\mr{b}}\mhm{X} \to \mr{D}\prns{X\prns{\mb{C}}^{\mr{an}}, \mb{Q}}$ (\cite{Saito_mixed-hodge}) and $\varrho^*_{\mr{Betti},X}: \mc{SH}\prns{X} \to \mr{D}\prns{X\prns{\mb{C}}^{\mr{an}}, \mb{Q}}$ (\cite{Ayoub_operations-de-Grothendieck}) compatible with Grothendieck's six functors.
One expects that the restriction of $\varrho^*_{\mr{Betti},X}$ to the full subcategory $\mc{SH}\prns{X}_{\aleph_0} \subseteq \mc{SH}\prns{X}$ spanned by the $\aleph_0$-presentable objects actually factors as the composite of a symmetric monoidal triangulated \emph{Hodge realization} functor $\varrho^*_{\mr{Hdg},X}: \mc{SH}\prns{X}_{\aleph_0} \to \mr{D}^{\mr{b}}\mhm{X}$ and $\omega^*_X$, and that this functor $\varrho^*_{\mr{Hdg},X}$ is itself compatible with Grothendieck's six functors.

Higher algebra, as developed in \cite{Lurie_higher-algebra}, offers an elegant approach to this open problem, consisting of two main ingredients:
$(i)$ a sufficiently refined construction of $\varrho^*_{\mr{Hdg},X}$ for the base case $X = \spec{\mb{C}}$, and $(ii)$ some general results concerning Grothendieck's six-functor formalism in the context of stable symmetric monoidal quasi-categories.
Our goal here is to address $(i)$ as follows, deferring discussion of $(ii)$ and the question of constructing Hodge realization functors over more general bases $X$ to a forthcoming preprint.

For $\bs{\Lambda} \in \brc{\mb{Z}, \mb{Q}}$, let $\mhs^{\mr{p}}_{\bs{\Lambda}}$ denote the category of polarizable mixed Hodge $\bs{\Lambda}$-structures and $\on{gr}\mhs^{\mr{p}}_{\bs{\Lambda}}$ the category of $\mb{Z}$-graded objects thereof.
We rectify P.~Deligne's functor $\mr{H}^{\bullet}_{\mr{Betti}}\prns{-, \mb{Z}}: \prns{\sch[ft]{\mb{C}}}\op \to \on{gr}\mhs^{\mr{p}}_{\mb{Z}}$, assigning to $X$ the graded polarizable mixed Hodge $\mb{Z}$-structure on its Betti cohomology (\cite{Deligne_hodgeIII}), to a presheaf of with values in a quasi-category of $\mb{E}_{\infty}$-algebras in the derived quasi-category of $\mhs^{\mr{p}}_{\bs{\Lambda}}$.
With rational coefficients, we further rectify this to a strict presheaf of commutative differential graded algebras in the ind-completion $\ind{\mhs^{\mr{p}}_{\mb{Q}}}$.

The issue of rectification of mixed Hodge structures considered here is of interest independent from the aforementioned motivic questions.
Indeed, the question has a long history: cf.~\cite[8.15]{Navarro-Aznar_hodge-deligne}, \cite[2.3.6]{Guillen-Navarro_critere-d'extension}, \cite[4.4]{Cirici-Guillen_E_1-formality}.
The fundamental component in such rectification results is always A.~Beilinson's equivalence of triangulated categories $\mr{D}^{\mr{b}}\prns{\mhs^{\mr{p}}_{\mb{Z}}} \simeq \mr{D}^{\mr{b}}_{\mc{H}^{\mr{p}}, \mb{Z}}$ from the bounded derived category of $\mhs^{\mr{p}}_{\mb{Z}}$ to the derived category of polarizable mixed Hodge $\mb{Z}$-complexes (\cite[3.11]{Beilinson_absolute-hodge}).
Below, we establish the following refinement of this equivalence with rational coefficients.

\begin{thmA}
Beilinson's equivalence $\mr{D}^{\mr{b}}\prns{\mhs^{\mr{p}}_{\mb{Q}}} \simeq \mr{D}^{\mr{b}}_{\mc{H}^{\mr{p}},\mb{Q}}$ can be promoted to an equivalence of symmetric monoidal quasi-categories.
\end{thmA}

\noindent
Combined with technical results about model structures on complexes in ind-completions of $\mb{Q}$-linear Tannakian categories established in \S1, this allows us to deduce the following rather strong rectification result.

\begin{thmB}
There is functor $\tilde{\Gamma}_{\mr{Hdg}}\prns{-, \mb{Q}}$ from $\prns{\sch[ft]{\mb{C}}}\op$ the category of commutative differential graded algebras in $\ind{\mhs^{\mr{p}}_{\mb{Q}}}$ such that, for each $X$, the cohomology of $\tilde{\Gamma}_{\mr{Hdg}}\prns{X, \mb{Q}}$ is naturally isomorphic to Deligne's mixed Hodge structure on $\mr{H}^{\bullet}_{\mr{Betti}}\prns{X, \mb{Q}}$.
\end{thmB}

\noindent
Gluing the functor $\tilde{\Gamma}_{\mr{Hdg}}\prns{-, \mb{Q}}$ of \ref{rect.6} with the ``singular cochain complex'' functor, we obtain the following variant with integral coefficients.

\begin{corC}
There is a functor $\tilde{\Gamma}_{\mr{Hdg}}\prns{-, \mb{Z}}$ from $\prns{\sch[ft]{\mb{C}}}\op$ to the quasi-category of $\mb{E}_{\infty}$-algebras in the symmetric monoidal derived quasi-category $\D{\ind{\mhs^{\mr{p}}_{\mb{Z}}}}$ such that, for each $X$, the cohomology of $\tilde{\Gamma}_{\mr{Hdg}}\prns{X, \mb{Z}}$ is naturally isomorphic to Deligne's mixed Hodge structure on $\mr{H}^{\bullet}_{\mr{Betti}}\prns{X, \mb{Z}}$.
\end{corC}

As mentioned above, our intended application of these results is the construction of Hodge realization functors $\varrho^*_{\mr{Hdg},X}$.
It is straightforward to check, using \ref{int.6} and M.~Robalo's universal property of the $\mb{P}^1$-stable $\mb{A}^1$-homotopy quasi-category (\cite[Corollary 1.2]{Robalo_K-theory-and-the-bridge}), that these results do indeed lead to the desired Hodge realization functor over $X = \spec{\mb{C}}$.
This construction actually provides a symmetric monoidal functor between symmetric monoidal quasi-categories, rather than a mere symmetric monoidal triangulated functor.
In this sense, we obtain a refined version of previous constructions due to A.~Huber (\cite{Huber_mixed-motives, Huber_realization-of-voevodsky's, Huber_corrigendum}), M.~Levine (\cite[2.3.10]{Levine_mixed-motives}), and F.~Lecomte and N.~Wach (\cite{Lecomte-Wach_realisation-de-hodge}).

Aside from playing a key role in the larger project of constructing $\varrho^*_{\mr{Hdg},X}$ for more general $X$, let us mention another application of our results.
Using \ref{rect.6} \resp{\ref{int.9}}, one can construct a motivic $\mb{E}_{\infty}$-ring spectrum $\mc{E}_{\mr{Hdg}}$ in $\mc{SH}\prns{\mb{C}}$ representing rational \resp{integral} absolute Hodge cohomology (\cite[\S5]{Beilinson_absolute-hodge}) and a morphism of $\mb{E}_{\infty}$-spectra $\mr{H}\mb{Z} \to \mc{E}_{\mr{Hdg}}$ from the motivic Eilenberg-Mac\thinspace Lane spectrum to $\mc{E}_{\mr{Hdg}}$ inducing regulator \resp{cycle-class} morphisms from the rational higher $\mr{K}$-theory \resp{higher Chow groups} of the smooth $\mb{C}$-scheme of finite type $X$ to its rational \resp{integral} absolute Hodge cohomology.

This motivic $\mb{E}_{\infty}$-spectrum $\mc{E}_{\mr{Hdg}}$ will be crucial to our approach to the construction of $\varrho^*_{\mr{Hdg}, X}$ for more general $X$:
higher algebra allows us to make sense of a well-behaved symmetric monoidal quasi-category of modules over $\mc{E}_{\mr{Hdg}}$ in $\mc{SH}\prns{X}$ and we show in a forthcoming work that this quasi-category of modules is naturally a full sub-quasi-category of the derived quasi-category $\D{\ind{\mhs^{\mr{p}}_{\mb{Z}}}}$.
Our strategy for bases $X$ of higher dimension is to generalize this result. 

\subsection*{Relation to other work}

In some form or other, the essential results of the first three sections below are contained in the author's $2013$ PhD thesis.
Interesting related work has appeared since then.

In \cite[Appendix A.2]{Pridham_tannaka-duality}, J.P.~Pridham discusses an construction related to \ref{rect.6}, applying different techniques and treating only the case of smooth $\mb{C}$-schemes.

Working with real coefficients, in \cite{Bunke-Nikolaus-Tamme_beilinson-regulator}, U.~Bunke, T.~Nikolaus and G.~Tamme have used similar techniques for lifting regulator morphisms to morphisms of motivic $\mb{E}_{\infty}$-ring spectra, further analyzing the structure of the motivic $\mb{E}_{\infty}$-ring spectrum representing real absolute Hodge cohomology and its relation to differential algebraic $\mr{K}$-theory.

An alternative construction of a motivic commutative ring spectrum representing absolute Hodge cohomology with real coefficients will be presented in the PhD thesis of A.~Navarro~Garamendia (\cite{Navarro_thesis}).

Using our rectification result, W.~Soergel and M.~Wendt have studied a motivic $\mb{E}_{\infty}$-ring spectrum, denoted by $\mc{E}_{\rm{GrH}}$ in \cite{Soergel-Wendt_perverse-motives}, which represents an interesting variant of absolute Hodge cohomology.

\subsection*{Organization}

We begin in \S1 by establishing some technical results on model structures and Tannakian categories, specifically showing that the injective model structure on the category of complexes of ind-objects in a $\mb{Q}$-linear Tannakian category is symmetric monoidal (\ref{tannakian.7}); that the category of commutative algebras in this symmetric monoidal category admits a model structure (\ref{tannakian.8}); and that this model structure on commutative algebras allows for a useful rectification result (\ref{tannakian.10}).

In \S2, working with rational coefficients, we construct a symmetric monoidal quasi-category of mixed Hodge complexes and use it to promote Beilinson's triangulated equivalence $\mr{D}^{\mr{b}}\prns{\mhs^{\mr{p}}_{\mb{Q}}} \simeq \mr{D}^{\mr{b}}_{\mc{H}^{\mr{p}},\mb{Q}}$ to an equivalence of symmetric monoidal quasi-categories (\ref{hodge.7}).

In \S3, we combine the results from the previous sections to construct the functor of \ref{rect.6} in two steps.
First, we restrict construct the functor after restricting the domain to separated, smooth $\mb{C}$-schemes (\ref{rect.2}).
Then, using a result of V.~Voevodsky, we extend the functor to all $\mb{C}$-schemes of finite type (\ref{rect.6}).

In \S4, we establish some general results on fiber products of stable quasi-categories and t-structures (\ref{int.2}, \ref{int.3}) and use them to show that the derived quasi-category of mixed Hodge $\mb{Z}$-structures is a fiber product of the derived quasi-categories of mixed Hodge $\mb{Q}$-structures and Abelian groups over the derived quasi-category of $\mb{Q}$-modules (\ref{int.5}).
We then establish the required functoriality of singular cochain complexes of associated analytic spaces (\ref{int.8}) and deduce \ref{int.9}.

\subsection*{Notation and conventions}

\noindent
\textsc{Grothendieck universes:}
We assume that each set is an element of a Grothendieck universe.
Fix uncountable Grothendieck universes $\mf{U} \in \mf{V}$ such that the categories $\set$, $\ab$ and $\cat$ of $\mf{U}$-sets, $\mf{U}$-small Abelian groups and $\mf{U}$-small categories are $\mf{V}$-small. 
Unless context dictates otherwise, all commutative rings and schemes will be $\mf{U}$-small.
We shall consider variations on such monstrosities as the category $\CAT$ of $\mf{V}$-small categories, which is not $\mf{V}$-small, but ambiguity is unlikely to result from our refusal to name a sufficiently large third Grothendieck universe.

\vspace{.25\baselineskip}

\noindent
\textsc{Quasi-categories:}
We freely employ the language of quasi-categories and higher algebra as developed in \cite{Lurie_higher-topos, Lurie_higher-algebra}.
For brevity, we contract the word ``quasi-category'' to ``\emph{qcategory}''.

\vspace{.25\baselineskip}

\noindent
\textsc{Categories as qcategories:}
We regard all categories as qcategories by tacitly taking their nerves.
As justification for this convention, observe that the nerve functor $\nerve{}: \cat \to \sset$ is right Quillen with respect to the model structure on $\cat$ whose weak equivalences and fibrations are the equivalences of categories and the isofibrations, respectively, and the Joyal model structure on $\sset$, and the functor induced between the qcategories underlying these model structures is fully faithful (\cite[2.8]{Joyal_notes-on-quasi-categories}).

\vspace{.25\baselineskip}

\noindent
\textsc{Functors and limits:}
We say that a functor $F: \mc{C} \to \mc{D}$ between qcategories is \emph{$\mf{U}$-continuous} \resp{\emph{$\mf{U}$-cocontinuous}} if it preserves $\mf{U}$-limits \resp{$\mf{U}$-colimits}, i.e., limits \resp{colimits} of $\mf{U}$-small diagrams.
We also write $F \dashv G$ to indicate that the functor $F: \mc{C} \to \mc{D}$ is left adjoint to the functor $G: \mc{D} \to \mc{C}$ (\cite[5.2.2.1]{Lurie_higher-topos}).

\vspace{.25\baselineskip}

\noindent
\textsc{Presentability:}
Let $\kappa$ denote an infinite regular $\mf{U}$-cardinal.
We preserve the terminology from the theory of $1$-categories (\cite{Adamek-Rosicky_locally-presentable}) and refer to an object $X$ of a qcategory $\mc{C}$ as \emph{$\kappa$-presentable} if $\map{X}{-}_{\mc{C}}: \mc{C} \to \spc{}$ preserves $\kappa$-filtered colimits, i.e., if it is ``$\kappa$-compact'' in the sense of \cite[5.3.4.5]{Lurie_higher-topos}.
We say that the qcategory $\mc{C}$ is \emph{locally $\mf{U}$-presentable} \resp{\emph{locally $\kappa$-presentable}} if it is ``presentable'' \resp{``$\kappa$-compactly generated''} in the sense of \cite[5.5.0.18, 5.5.7.1]{Lurie_higher-topos}.

\vspace{.25\baselineskip}

\noindent
\textsc{Localization:}
If $\mc{C}$ is a $\mf{U}$-small qcategory and $\mf{W}$ a class of morphisms of $\mc{C}$, then there exists a functor $\lambda: \mc{C}  \to \mc{C}\brk{\mf{W}^{-1}}$ with the universal property that, for each $\mf{U}$-small qcategory $\mc{D}$, composition with $\lambda$ induces a fully faithful functor $\fun{\mc{C}\brk{\mf{W}^{-1}}}{\mc{D}} \hookrightarrow \fun{\mc{C}}{\mc{D}}$ whose essential image is spanned by those functors that send each element of $\mf{W}$ to an equivalence in $\mc{D}$ (\cite[1.3.4.2]{Lurie_higher-algebra}).
We refer to $\lambda$ or, abusively, $\mc{C}\brk{\mf{W}^{-1}}$, as a \emph{localization of $\mc{C}$ with respect to $\mf{W}$}.
If $\lambda$ admits a fully faithful right adjoint $\iota$, then we say that $\lambda$ is a \emph{reflective localization of $\mc{C}$}.
This applies in particular to the locally presentable setting:
if $\mc{C}$ is a locally $\mf{U}$-presentable category and $S$ is a $\mf{U}$-set of morphisms of $\mc{C}$, then the localization $\lambda: \mc{C} \to \mc{C} \brk{S^{-1}}$ is reflective, the essential image of its right adjoint is the full subqcategory spanned by the $S$-local objects, i.e., the objects $X \in \mc{C}$ such that, for each $f \in S$, the morphism $\map{f}{X}_{\mc{C}}$ is a weak homotopy equivalence, and $\mc{C}\brk{S^{-1}}$ is locally $\mf{U}$-presentable (\cite[5.5.4.15, 5.5.4.20]{Lurie_higher-topos}).

\vspace{.25\baselineskip}

\noindent
\textsc{Symmetric monoidal qcategories:}
A symmetric monoidal qcategory is, by definition (\cite[2.0.0.7]{Lurie_higher-algebra}), a coCartesian fibration $p: \mc{C}^{\otimes} \to \fin$ such that the morphisms $\rho^i: \ang{n} \to \ang{1}$ given by 
$
\rho^i\prns{j}
:=
1$ if $i = j$ and 
$
\rho^i\prns{j} = *$ if
$i \neq j$
induce functors $\rho^i_!: \mc{C}^{\otimes}_{\ang{n}} \to \mc{C}^{\otimes}_{\ang{1}}$ which in turn induce equivalences $\mc{C}^{\otimes}_{\ang{n}} \simeq \prns{\mc{C}^{\otimes}_{\ang{1}}}^n$.
We systematically suppress the fibration $p$ from the notation, referring to ``the symmetric monoidal qcategory $\mc{C}^{\otimes}$''.
We also refer to $\mc{C} := \mc{C}^{\otimes}_{\ang{1}}$ as the \emph{qcategory underlying $\mc{C}^{\otimes}$}.
Similarly, we use the notation $F^{\otimes}: \mc{C}^{\otimes} \to \mc{D}^{\otimes}$ for a possibly lax symmetric monoidal functor and $F: \mc{C} \to \mc{D}$ for the underlying functor.

Appealing to \cite[2.4.2.6]{Lurie_higher-algebra}, the category $\calg{\qcat^{\times}}$ of commutative algebra objects of $\qcat^{\times}$ is a convenient model for the qcategory of $\mf{U}$-small symmetric monoidal qcategories:
its objects correspond to $\mf{U}$-small symmetric monoidal qcategories and its morphisms to symmetric monoidal functors.

\vspace{.25\baselineskip}

\noindent
\textsc{Qcategories underlying model categories:}
The model categories appearing in the sequel will prove to be \emph{$\mf{U}$-combinatorial}, i.e., cofibrantly generated model categories whose underlying categories are locally $\mf{U}$-presentable (\cite[1.8]{Beke_sheafifiable-homotopy}, \cite[\S A.2.6]{Lurie_higher-topos}, \cite[1.21]{Barwick_left-and-right}).
Many will even prove to be \emph{$\mf{U}$-tractable} model categories, i.e., $\mf{U}$-combinatorial model categories whose generating cofibrations and trivial cofibrations have cofibrant domains (\cite[1.21]{Barwick_left-and-right}).

If $\mb{M}$ is a $\mf{U}$-combinatorial model category and $\mf{W}$ is its class of weak equivalences, then its \emph{underlying qcategory $\mb{M}\brk{\mf{W}^{-1}}$} is locally $\mf{U}$-presentable (\cite[1.3.4.15, 1.3.4.16]{Lurie_higher-algebra}).
If $\mb{M}^{\otimes}$ is a symmetric monoidal $\mf{U}$-combinatorial model category, then it admits an \emph{underlying locally $\mf{U}$-presentable symmetric monoidal qcategory $\mb{M}\brk{\mf{W}^{-}}\scr{\otimes}{}$} (\cite[4.1.3.6, 4.1.4.8]{Lurie_higher-algebra}).
Strictly speaking, the qcategory underlying the symmetric monoidal qcategory $\mb{M}\brk{\mf{W}^{-1}}\scr{\otimes}{}$ is defined to be the full subqcategory of $\mb{M}\brk{\mf{W}^{-1}}$ spanned by the cofibrant objects of $\mb{M}$, but the inclusion of this full subqcategory is an equivalence as each object of $\mb{M}$ is weakly equivalent to a cofibrant object.

\vspace{\baselineskip}

\noindent
\textsc{Notation:}
While we maintain most of the notations of \cite{Lurie_higher-topos,Lurie_higher-algebra}, the following list specifies the notable deviations therefrom and other frequently recurring symbols.

\renewcommand{\arraystretch}{1.2}

\begin{longtable}[c]{p{1.1in}>{\raggedright\arraybackslash}p{4.5in}}
$\mc{C}_{\aleph_0}$ & the full subqcategory of $\mc{C}$ spanned by the $\aleph_0$-presentable objects (\cite[5.3.4.5]{Lurie_higher-topos})
\\
$\mc{C}^{\amalg}$ & the coCartesian symmetric monoidal structure on the qcategory $\mc{C}$ (\cite[\S2.4.3]{Lurie_higher-algebra})
\\
$\mc{C}^{\times}$ & the Cartesian symmetric monoidal structure on the qcategory $\mc{C}$ (\cite[2.4.1.1]{Lurie_higher-algebra})
\\
$\mc{C}\brk{\mf{W}^{-1}}$ & a localization of the qcategory $\mc{C}$ with respect to the class of morphisms $\mf{W}$ (\cite[1.3.4.1]{Lurie_higher-algebra}) 
\\
$\ab$ & the category of $\mf{U}$-small Abelian groups
\\
$\calg{\mc{C}^{\otimes}}$ & the qcategory of commutative algebras in the symmetric monoidal qcategory $\mc{C}^{\otimes}$ (\cite[2.1.3.1]{Lurie_higher-algebra})
\\
$\fun{\mc{C}^{\otimes}}{\mc{D}^{\otimes}}^{\otimes}$ & the qcategory of symmetric monoidal functors $F^{\otimes}: \mc{C}^{\otimes} \to \mc{D}^{\otimes}$ (\cite[2.1.3.7]{Lurie_higher-algebra}) 
\\
$\h^r$ & the degree $r$ cohomology functor $\trun^{\leq r}\trun^{\geq r}: \mc{C} \to \mc{C}^{\heartsuit}$ of a $\trun$-structure on the stable qcategory $\mc{C}$ (\cite[1.2.1.4]{Lurie_higher-algebra})
\\
$\ho{\mc{C}}$ & the homotopy category of the qcategory $\mc{C}$ (\cite[1.2.3]{Lurie_higher-topos})
\\
$\ind{\mc{C}}$ & the ind-completion of the qcategory $\mc{C}$ (\cite[5.3.5.1]{Lurie_higher-topos})
\\
$\map{X}{Y}_{\mc{C}}$ & the mapping space between two objects $X$ and $Y$ of the qcategory $\mc{C}$ (\cite[1.2.2]{Lurie_higher-topos})
\\
$\mod{\mc{C}}_A$ & the qcategory of modules over $A \in \calg{\mc{C}^{\otimes}}$ (\cite[4.5.1.1]{Lurie_higher-algebra})
\\
$\nerve{\mc{C}}$ & the nerve of the category $\mc{C}$
\\
$\dgnerve{\bs{\mc{C}}}$, $\altdgnerve{\bs{\mc{C}}}$ & two constructions of the differential graded nerve of the differential graded category $\bs{\mc{C}}$ (\cite[1.3.1.6, 1.3.1.16]{Lurie_higher-algebra})
\\
$\snerve{\bs{\mc{C}}}$ & the simplicial nerve of the simplicial category $\bs{\mc{C}}$ (\cite[1.1.5.5]{Lurie_higher-topos})
\\
$\psh{\mc{C}}{\mc{D}}$ & $\fun{\mc{C}\op}{\mc{D}}$
\\
$\qcat$ & the qcategory of $\mf{U}$-small \resp{$\mf{V}$-small} qcategories (\cite[3.0.0.1]{Lurie_higher-topos})
\\
$\qcat^{\mr{Ex}}$, $\QCAT^{\mr{Ex}}$ & the qcategory of $\mf{U}$-small \resp{$\mf{V}$-small} stable qcategories and exact functors (\cite[\S1.1.4]{Lurie_higher-algebra})
\\
$\qcat^{\times}$
& the Cartesian symmetric monoidal qcategories of $\mf{U}$-small \resp{$\mf{V}$-small} qcategories (\cite[2.4.1.5]{Lurie_higher-algebra}) 
\\
$\sch[ft]{S}$ & the essentially $\mf{U}$-small category of $S$-schemes of finite type
\\
$\sm[ft]{S}$, $\sm[sft]{S}$ & the full subcategory of $\sch[ft]{S}$ spanned by the smooth \resp{smooth and separated} $S$-schemes
\\
$\sset$ & the category of simplicial $\mf{U}$-sets
\\
$\spc{}$ & the qcategory of $\mf{U}$-small spaces, i.e., the qcategory underlying the model structure on $\sset$ whose weak equivalences and fibrations are the weak homotopy equivalences and the Kan fibrations, respectively (\cite[1.2.16.1]{Lurie_higher-topos})
\\
$\spc{}_*$ & the qcategory of pointed objects in $\spc{}$ (\cite[4.8.1.20]{Lurie_higher-algebra})
\\
$\trun^{\leq r}$, $\trun^{\geq r}$ & the truncations of a cohomological $\trun$-structure on a stable qcategory $\mc{C}$
\\
$\mf{U} \in \mf{V}$ & fixed Grothendieck universes
\\
$\yon{}{}$ &
the Yoneda embedding (\cite[\S5.1.3]{Lurie_higher-topos})
\end{longtable}

\renewcommand{\arraystretch}{1}

\section{Deriving Tannakian categories}

\begin{motivation}
One nice property of combinatorial model structures\textemdash among many others\textemdash is that they allow for convenient rectification results.
For instance, by \cite[1.3.4.25]{Lurie_higher-algebra}, if $\mb{M}$ is a $\mf{U}$-combinatorial model category, $\mf{W}$ its class of weak equivalences and $\mc{C}$ a $\mf{U}$-small category, then any functor $F: \mc{C} \to \mb{M}\brk{\mf{W}^{-1}}$ can be rectified to a functor $F': \mc{C} \to \mb{M}$, i.e., there exists a functor $F': \mc{C} \to \mb{M}$ whose composite with the localization $\mb{M} \to \mb{M}\brk{\mf{W}^{-1}}$ is equivalent to $F$.
In a similar vein but under more restrictive hypotheses, if $\mb{M}^{\otimes}$ is a $\mf{U}$-combinatorial symmetric monoidal model category, then commutative algebras in the underlying symmetric monoidal qcategory $\mb{M}\brk{\mf{W}^{-1}}\scr{\otimes}{}$ can be rectified to commutative algebras in $\mb{M}^{\otimes}$ (\cite[4.5.4.7]{Lurie_higher-algebra}).
Our goal in this section is to show that the category of complexes of ind-objects in a Tannakian category admits a $\mf{U}$-combinatorial model structure allowing for both of these types of rectifications.
\end{motivation}

\begin{summary}
After a brief review of the theory of Tannakian categories, we show that, for $\mb{T}$ a $\mf{U}$-small Tannakian category, the categories $\cpx{\ind{\mb{T}}}$ and $\calg{\cpx{\ind{\mb{T}}}\scr{\otimes}{}}$ admit $\mf{U}$-combinatorial model structures (\ref{tannakian.7}, \ref{tannakian.8}).
Theorem \ref{tannakian.10} shows that commutative algebras in $\D{\ind{\mb{T}}}\scr{\otimes}{}$ can be rectified to commutative algebras in $\cpx{\ind{\mb{T}}}\scr{\otimes}{}$.
\end{summary}

\begin{defn}
\label{tannakian.1}
Let $\mb{A}$ be an Abelian category.
\begin{enumerate}
\item
As in \cite[3.1]{Drew_verdier-quotients}, we construct its \emph{bounded derived qcategory} $\D{\mb{A}}^{\mr{b}}$ as follows:
take the differential graded nerve $\K{\mb{A}}^{\mr{b}} := \dgnerve{\cpx{\mb{A}}^{\mr{b}}}$ of the differential graded category of bounded cochain complexes in $\mb{A}$ (\cite[1.3.1.6]{Lurie_higher-algebra}) and then take the Verdier quotient $\D{\mb{A}}^{\mr{b}} := \K{\mb{A}}^{\mr{b}}/\on{\mc{A}c}\prns{\mb{A}}$ with respect to the full subqcategory $\on{\mc{A}c}\prns{\mb{A}}$ spanned by the acyclic complexes.
The analogous construction for unbounded complexes results in the \emph{unbounded derived qcategory} $\D{\mb{A}}$ of $\mb{A}$.
\item
The category $\mb{A}$ is \emph{$\mf{U}$-Grothendieck Abelian} if it is an Abelian, locally $\mf{U}$-presentable category in which $\aleph_0$-filtered $\mf{U}$-colimits preserve finite limits.
By \cite[3.10]{Beke_sheafifiable-homotopy}, this is equivalent to the classical definition as an AB5 category with a generator.
If $\mb{A}$ is essentially $\mf{U}$-small, then its ind-completion $\ind{\mb{A}}$ is $\mf{U}$-Grothendieck Abelian.
If $\mb{A}$ is $\mf{U}$-Grothendieck Abelian, then $\cpx{\mb{A}}$ admits a $\mf{U}$-combinatorial model structure (\cite[A.2.6.1]{Lurie_higher-topos}) whose cofibrations and weak equivalences are the monomorphisms and quasi-isomorphisms, respectively (\cite[3.13]{Beke_sheafifiable-homotopy}), called the \emph{injective model structure} and denoted by $\cpx{\mb{A}}\scr{}{\mr{inj}}$.
Its homotopy category is the unbounded derived category of $\mb{A}$.
Let $\D{\mb{A}}$ denote the stable, locally $\mf{U}$-presentable qcategory underlying $\cpx{\mb{A}}\scr{}{\mr{inj}}$ (\cite[1.3.4.22]{Lurie_higher-algebra}).
\item
If $\mb{A}^{\otimes}$ is a symmetric monoidal structure on $\mb{A}$, then $\cpx{\mb{A}}^{\mr{b}}$ and $\cpx{\mb{A}}$ inherit symmetric monoidal structures given informally by
\begin{equation}
\label{tannakian.1.1}
\prns{K \otimes L}^n 
:= 
\bigoplus_{r \in \mb{Z}} \prns{K^r \otimes L^{n-r}},
\qquad
\d\prns{x \otimes y}
:=
\d\prns{x} \otimes y + \prns{-1}^{\deg\prns{x}}x \otimes \d\prns{y}.
\end{equation}
By \cite[1.3.4.5, 4.1.3.4]{Lurie_higher-algebra}, $\K{\mb{A}}^{\mr{b}}$ and $\K{\mb{A}}$ inherit symmetric monoidal structures from $\cpx{\mb{A}}^{\mr{b}}\scr{\otimes}{}$ and $\cpx{\mb{A}}\scr{\otimes}{}$, respectively.
\item
If $\mb{A}^{\otimes}$ is a symmetric monoidal structure on $\mb{A}$ such that $\prns{-} \otimes \prns{-}$ is exact separately in each variable, then, by \cite[3.2]{Drew_verdier-quotients}, the Verdier quotient functor $q: \K{\mb{A}}^{\mr{b}} \to \D{\mb{A}}^{\mr{b}}$ underlies a symmetric monoidal functor $q^{\otimes}$ realizing $\D{\mb{A}}^{\mr{b}}\scr{\otimes}{}$ as a \emph{symmetric monoidal Verdier quotient of $\K{\mb{A}}^{\mr{b}}\scr{\otimes}{}$ by $\on{\mc{A}c}\prns{\mb{A}}$} in the sense of \cite[1.5]{Drew_verdier-quotients}.
This means that, for each stable symmetric monoidal qcategory $\mc{C}^{\otimes}$, composition with $q^{\otimes}$ induces a fully faithful functor $\fun{\D{\mb{A}}^{\mr{b}}\scr{\otimes}{}}{\mc{C}^{\otimes}}^{\otimes} \hookrightarrow \fun{\K{\mb{A}}^{\mr{b}}\scr{\otimes}{}}{\mc{C}^{\otimes}}^{\otimes}$ whose essential image is spanned by the symmetric monoidal functors sending each object of $\on{\mc{A}c}\prns{\mb{A}}$ to a zero object.
\end{enumerate}
\end{defn}

\begin{defn}
\label{tannakian.2}
An essentially $\mf{U}$-small closed symmetric monoidal category $\mb{T}^{\otimes}$ is \emph{Tannakian} if it satisfies the following conditions:
\begin{enumerate}
\item
$\mb{T}$ is Abelian;
\item
$\hom{\1{\mb{T}}}{\1{\mb{T}}}_{\mb{T}}$ is a field of characteristic zero;
\item
each object of $\mb{T}$ is $\otimes$-dualizable \textup{(\cite[4.6.1.12]{Lurie_higher-algebra})};
\item
for each $V \in \mb{T}$, the composite $\1{\mb{T}} \xrightarrow{\eta} V \otimes V^{\vee} \simeq V^{\vee} \otimes V \xrightarrow{\varepsilon} \1{\mb{T}}$ is a nonnegative integer under the identification of $\mb{Z}$ with its image in the field $\hom{\1{\mb{T}}}{\1{\mb{T}}}_{\mb{T}}$ of characteristic zero, where $\eta$ and $\varepsilon$ are the coevaluation and evaluation morphisms, respectively, and the equivalence $V \otimes V^{\vee} \simeq V^{\vee} \otimes V$ is the symmetry isomorphism.
\end{enumerate}
\end{defn}

\begin{rmk}
\label{tannakian.3}
This definition is more restrictive than the original one of \cite[III, 3.2.1]{Saavedra_categories-tannakiennes}, but the two are equivalent once we require $\hom{\1{\mb{T}}}{\1{\mb{T}}}_{\mb{T}}$ to be a field of characteristic zero by \cite[7.1]{Deligne_categories-tannakiennes}.
\end{rmk}

\begin{ex}
\label{tannakian.4}
Let $\KK$ be a field of characteristic zero.
\begin{enumerate}
\item 
If $G$ is an affine gerbe on the $\mr{fpqc}$-site $\prns{\sch[ft]{\KK}}_{\mr{fpqc}}$ of $\KK$-schemes of finite type,
then the essentially $\mf{U}$-small category $\qcoh{G}^{\vee}$ of locally free quasi-coherent sheaves of finite rank on $G$ is Tannakian when equipped with the usual tensor product of quasi-coherent sheaves.
By \cite[1.12]{Deligne_categories-tannakiennes}, every essentially $\mf{U}$-small Tannakian category $\mb{T}$ arises in this way. 
\item 
More concretely, the category $\mhs_{\KK}$ of mixed Hodge $\KK$-structures (\cite[2.3.8]{Deligne_hodgeII}) is Tannakian, as is the full subcategory $\mhs^{\mr{p}}_{\KK} \subseteq \mhs_{\KK}$ spanned by the objects $\prns{H, W, F}$ such that the pure Hodge $\KK$-structure on $\on{gr}^W_n(H)$ induced by the filtration $F$ admits a polarization (\cite[2.1.15]{Deligne_hodgeII}) for each $n\in\mb{Z}$.
\end{enumerate}
\end{ex}

\begin{rmk}
\label{tannakian.5}
Let $\mb{T}^{\otimes}$ be an essentially $\mf{U}$-small Tannakian category.
\begin{enumerate}
\item
By \cite[7.1]{Deligne_categories-tannakiennes}, there exists a field extension $\KK := \hom{\1{\mb{T}}}{\1{\mb{T}}}_{\mb{T}} \hookrightarrow \KK'$ and a $\KK$-linear, exact symmetric monoidal functor $\omega^{\otimes}: \mb{T}^{\otimes} \to \mod{\ab}_{\KK'}\scr{\otimes}{}$.
By \cite[1.19]{Deligne-Milne_tannakian-categories}, $\omega$ is faithful.
Using the exactness and faithfulness of $\omega$, one finds that $\mb{T}$ is Noetherian, since $\mod{\ab}_{\KK'}$ is.
By \emph{Noetherian}, we mean that each family of subobjects of each fixed object $V \in \mb{T}$ contains a maximal element.
\item
If $\mb{A}$ is an Abelian category, we define \emph{homological dimension of $A \in \mb{A}$} to be $\on{hdim}\prns{A} := \sup\brc{n \in \mb{Z}_{\geq0} \mid \exists B \in \mb{A} \brk{\ext{A}{B}^n_{\mb{A}}}} \in \mb{Z}_{\geq 0} \cup \brc{\infty}$ and we define the \emph{homological dimension of $\mb{A}$} to be $\on{hdim}\prns{\mb{A}} := \sup\brc{\on{hdim}\prns{A} \mid A \in \mb{A}} \in \mb{Z}_{\geq0} \cup \brc{\infty}$.
Since each $V \in \mb{T}$ is $\otimes$-dualizable, the adjunction $\prns{-} \otimes V \dashv V^{\vee} \otimes \prns{-}$ shows that $\on{hdim}\prns{\mb{T}} = \on{hdim}\prns{\1{\mb{T}}}$.
As a consequence, if $\on{hdim}\prns{\1{\mb{T}}} < \infty$, then $\mb{T}$ satisfies the hypotheses of \cite[4.7]{Drew_verdier-quotients} and it follows that the $\aleph_0$-presentable objects of $\D{\ind{\mb{T}}}$ are precisely the $\otimes$-dualizable objects and the natural symmetric monoidal functor $\ind{\D{\mb{T}}^{\mr{b}}}\scr{\otimes}{} \to \D{\ind{\mb{T}}}\scr{\otimes}{}$ is an equivalence.
This applies in particular to $\mb{T}^{\otimes} = \mhs^{\mr{p}, \otimes}_{\KK}$, since $\on{hdim}\prns{\mhs^{\mr{p}}_{\KK}} = 1$ by \cite[3.35]{Peters-Steenbrink_mixed-hodge}.
\end{enumerate}
\end{rmk}

\begin{lemma}
\label{tannakian.6}
Let $\KK$ be a field of characteristic zero, $F^{\otimes}: \mb{T}^{\otimes} \to \mb{T}'^{\otimes}$ a $\KK$-linear, exact symmetric monoidal functor between two essentially $\mf{U}$-small $\KK$-linear Tannakian categories.
Then:
\begin{enumerate}
\item
$\cpx{\ind{F}}: \cpx{\ind{\mb{T}}}\scr{\otimes}{\mr{inj}} \to \cpx{\ind{\mb{T}'}}\scr{\otimes}{\mr{inj}}$ is a $\KK$-linear, exact, faithful symmetric monoidal left Quillen functor;
\item
there is an essentially commutative square
\[
\xymatrix{
\mb{T}
\ar[rr]_-{F}
\ar@{^{(}->}[d]
&
&
\mb{T}'
\ar@{^{(}->}[d]
\\
\cpx{\ind{\mb{T}}}
\ar[rr]^-{\cpx{\ind{F}}}
&
&
\cpx{\ind{\mb{T}'}}
}
\]
in which vertical arrows are the evident inclusions in degree zero; and
\item
the functor $\D{\ind{F}}: \D{\ind{\mb{T}}} \to \D{\ind{\mb{T}'}}$ is conservative and $\trun$-exact with respect to the natural $\trun$-structures.
\end{enumerate}
\end{lemma}

\begin{proof}
By \cite[1.19]{Deligne-Milne_tannakian-categories}, $F$ is faithful.
The existence of $\cpx{\ind{F}}$, as well as its $\KK$-linearity, exactness, faithfulness, $\mf{U}$-cocontinuity and compatibility with the symmetric monoidal structures, follows from \cite[Expos\'e I, 8.9.8, 8.6.4]{SGA4a} and the obvious functorial properties of complexes and categories of ind-objects.
The Adjoint Functor Theorem (\cite[1.66]{Adamek-Rosicky_locally-presentable}) implies that the $\mf{U}$-cocontinuous functor $\cpx{\ind{F}}$ is a left adjoint.
As $\cpx{\ind{F}}$ preserves quasi-isomorphisms and monomorphisms, it is left Quillen with respect to the injective model structures.
This proves $(i)$, and $(ii)$ is obvious.

As $\cpx{\ind{F}}$ preserves quasi-isomorphisms, it induces $\D{\ind{F}}: \D{\ind{\mb{T}}} \to \D{\ind{\mb{T}'}}$ by the universal property of the localization.
The derived functor of an exact functor between Abelian categories is $\trun$-exact with respect to the natural $\trun$-structures, so $\D{\ind{F}}$ is $\trun$-exact.
Let us check that $\D{\ind{F}}$ is conservative.
As $\D{\ind{F}}$ is an exact functor between stable qcategories, it suffices to check that it reflects zero objects.
The natural $\trun$-structure on the derived qcategory of an Abelian category is nondegenerate, so it suffices to show that, for each object $K$ of the heart $\D{\ind{\mb{T}}}\scr{\heartsuit}{}$, $\D{\ind{F}}\prns{K} = 0$ implies $K = 0$.
In this case, we may assume $K$ is concentrated in degree zero, given by an object $V$ of $\mb{T}$.
As $\ind{F}$ is faithful, we have a commutative square
\[
\xymatrix{
\pi_0\map{K}{K}_{\D{\ind{\mb{T}}}}
\ar[r]^-{\sim}
\ar[d]_-{\D{\ind{F}}}
&
\hom{V}{V}_{\ind{\mb{T}}}
\ar@{^{(}->}[d]^-{\ind{F}}
\\
\pi_0\map{\D{\ind{F}}\prns{K}}{\D{\ind{F}}\prns{K}}_{\D{\ind{\mb{T}'}}}
\ar[r]^-{\sim}
&
\hom{\ind{F}\prns{V}}{\ind{F}\prns{V}}_{\ind{\mb{T}'}}
}
\]
in which the vertical arrow on the right is injective and $(iii)$ follows.
\end{proof}

\begin{prop}
\label{tannakian.7}
Let $\mb{T}^{\otimes}$ be a $\mf{U}$-small Tannakian category.
Then $\cpx{\ind{\mb{T}}}\scr{\otimes}{\mr{inj}}$ is a left proper, stable, $\mf{U}$-tractable symmetric monoidal model category satisfying the monoid axiom.
\end{prop}

\begin{proof}
The $\mf{U}$-tractability follows from the remark that all objects is cofibrant.
The stability follows from \cite[1.3.4.24, 1.4.2.27]{Lurie_higher-algebra}.
Let us show that $\cpx{\ind{\mb{T}}}\scr{\otimes}{\mr{inj}}$ is a symmetric monoidal model category.
Choose $\omega^{\otimes}: \mb{T}^{\otimes} \to \mod{\ab}_{\KK'}\scr{\otimes}{}$ as in \ref{tannakian.5}$(i)$.
As the canonical $\mf{U}$-cocontinuous symmetric monoidal functor $\ind{\mod{\ab}_{\KK'}\scr{}{\aleph_0}}\scr{\otimes}{} \to \mod{\ab}_{\KK'}\scr{\otimes}{}$ is an equivalence, $\omega^{\otimes}$ induces a $\KK$-linear, exact, faithful, $\mf{U}$-cocontinuous symmetric monoidal functor $\cpx{\ind{\mb{T}}}\scr{\otimes}{} \to \cpx{\mod{\ab}_{\KK'}}\scr{\otimes}{}$ by \ref{tannakian.6}, which we abusively denote by $\omega^{\otimes}$.

As $\1{\mb{T}}$ is cofibrant, it remains to establish the pushout-product axiom.
Let $f: K \to K'$ and $g: L \to L'$ be two cofibrations of $\cpx{\ind{\mb{T}}}\scr{}{\mr{inj}}$, i.e., two monomorphisms.
We claim that the canonical morphism $f \mathrel{\Box} g: \prns{K \otimes L'} \amalg_{K \otimes L} \prns{K' \otimes L} \to K' \otimes L'$ is a monomorphism, and that it is moreover a quasi-isomorphism if $f$ or $g$ is.
The image of $f \mathrel{\Box} g$ under $\omega$ is the morphism $\omega\prns{f} \mathrel{\Box} \omega\prns{g}$.
Faithful, exact functors preserve and reflect monomorphisms, so the pushout-product axiom in $\cpx{\mod{\ab}_{\KK'}}\scr{\otimes}{\mr{inj}}$, which holds by \cite[2.3]{Drew_verdier-quotients}, implies that $f \mathrel{\Box} g$ is a monomorphism.
Similarly, the faithful, exact functor $\omega$ preserves and reflects quasi-isomorphisms, so if $f$ or $g$ is a quasi-isomorphism, then $\omega\prns{f}$ or $\omega\prns{g}$ is.
This implies that $\omega\prns{f} \mathrel{\Box} \omega\prns{g}$, and hence also $f \mathrel{\Box} g$, is a quasi-isomorphism.
Hence, $\cpx{\ind{\mb{T}}}\scr{}{\mr{inj}}$ is a symmetric monoidal model category.
Since each object is cofibrant, it is left proper and satisfies the monoid axiom (\cite[3.4]{Schwede-Shipley_algebras-and-modules}).
\end{proof}

\begin{lemma}
\label{tannakian.8}
Let $\mb{T}^{\otimes}$ be a $\mf{U}$-small Tannakian category.
The category $\calg{\cpx{\ind{\mb{T}}}\scr{\otimes}{}}$ admits a $\mf{U}$-combinatorial model structure whose weak equivalences \resp{fibrations} are the morphisms inducing quasi-isomorphisms \resp{fibrations} between the underlying objects of $\cpx{\ind{\mb{T}}}\scr{}{\mr{inj}}$.
\end{lemma}

\begin{proof}
By \cite[3.2.3.5]{Lurie_higher-algebra} and \ref{tannakian.7}, $\calg{\cpx{\ind{\mb{T}}}\scr{\otimes}{}}$ is locally $\mf{U}$-presentable.
It therefore suffices to construct a cofibrantly generated model structure with the prescribed weak equivalences and fibrations.
Let $f: K \to L$ be a morphism of $\cpx{\ind{\mb{T}}}$.
We have the pushout-product morphism $f^{\Box 2} := f \mathrel{\Box} f: \prns{K \otimes L} \amalg_{K \otimes K} \prns{L \otimes K} \to L \otimes L$.
Iterating, we obtain morphisms $f^{\Box n}$ for $n \in \mb{Z}_{\geq0}$, which are $\mf{S}_n$-equivariant with respect to the $\mf{S}_n$-actions permuting factors of the tensor products appearing in the domain and codomain.
We thus regard $f^{\Box n}$ as a morphism of $\cpx{\ind{\mb{T}}}\scr{\mf{S}_n}{}$, the category of functors from the groupoid $\mf{S}_n$ into $\cpx{\ind{\mb{T}}}$.
The functor $\cpx{\ind{\mb{T}}} \to \cpx{\ind{\mb{T}}}\scr{\mf{S}_n}{}$ sending $K$ to itself with the trivial $\mf{S}_n$-action admits a left adjoint, the \emph{$\mf{S}_n$-coinvariants} functor $\prns{-}/\mf{S}_n$.

By \cite[3.2]{White_model-structures}, the existence of a cofibrantly generated model structure on the category $\calg{\cpx{\ind{\mb{T}}}\scr{\otimes}{}}$ with the prescribed weak equivalences will follow if we show that $\cpx{\ind{\mb{T}}}\scr{\otimes}{}$ satisfies the \emph{commutative monoid axiom} (\cite[3.1]{White_model-structures}):
for each trivial cofibration $f: K \to L$ of $\cpx{\ind{\mb{T}}}\scr{\otimes}{}$ and each $n \in \mb{Z}_{>0}$, the morphism $f^{\Box n}/\mf{S}_n$ is a trivial cofibration in $\cpx{\ind{\mb{T}}}\scr{}{\mr{inj}}$.

As in \ref{tannakian.5}, choose a $\KK$-linear, faithful, exact symmetric monoidal functor $\omega^{\otimes}: \mb{T}^{\otimes} \to \mod{\ab}_{\KK'}\scr{\otimes}{\aleph_0}$ for some field extension $\hom{\1{\mb{T}}}{\1{\mb{T}}}_{\mb{T}} \hookrightarrow \KK'$, let $\overline{\omega}^{\otimes} := \cpx{\ind{\omega}}\scr{\otimes}{}$, and let $f$ be a trivial cofibration of $\cpx{\ind{\mb{T}}}$ and $n \in \mb{Z}_{> 0}$.
By \ref{tannakian.6}, $\overline{\omega}$ reflects trivial cofibrations, so it suffices to show that $\overline{\omega}\prns{f^{\Box n}/\mf{S}_n}$ is a trivial cofibration.
On the other hand, $\overline{\omega}$ also preserves trivial cofibrations by \ref{tannakian.6}, so $\overline{\omega}f$ is a trivial cofibration.
Since $\overline{\omega}^{\otimes}$ is symmetric monoidal and $\mf{U}$-cocontinuous we have $\overline{\omega}\prns{f^{\Box n}/\mf{S}_n} \simeq \prns{\overline{\omega}f}^{\Box n}/\mf{S}_n$.
The claim therefore follows from the fact that $\cpx{\mod{\ab}_{\KK'}}\scr{\otimes}{}$ is freely powered (\cite[7.1.4.7]{Lurie_higher-algebra}), hence $\prns{\overline{\omega}f}^{\Box n}$ is a projective trivial cofibration in $\cpx{\mod{\ab}_{\KK'}}\scr{\mf{S}_n}{}$.
Indeed, as $\prns{-}/\mf{S}_n$ is left Quillen with respect to the projective model structure on its domain, $\prns{\overline{\omega}f}^{\Box n}/\mf{S}_n$ is a trivial cofibration, as desired.
\end{proof}

\begin{lemma}
\label{tannakian.9}
Consider the following data:
\begin{enumerate}
\item
$\mb{T}^{\otimes}$, a $\mf{U}$-small Tannakian category;
\item
$\KK := \hom{\1{\mb{T}}}{\1{\mb{T}}}_{\mb{T}} \hookrightarrow \KK'$, a field extension;
\item
$\omega^{\otimes}: \mb{T}^{\otimes} \to \mod{\ab}_{\KK'}\scr{\otimes}{\aleph_0}$, a $\KK$-linear exact symmetric monoidal functor;
\item
$\mf{W}$ \resp{$\mf{W}'$}, the class of weak equivalences in the model structure of \textup{\ref{tannakian.8}} on the category $\calg{\cpx{\ind{\mb{T}}}\scr{\otimes}{}}$ \resp{$\calg{\cpx{\mod{\ab}_{\KK'}}\scr{\otimes}{}}$}; and
\item
$\mc{C}$, a $\mf{U}$-small category.
\end{enumerate}
If the forgetful functor $\psi': \calg{\cpx{\mod{\ab}_{\KK'}}\scr{\otimes}{}} \brk{\mf{W}'^{-1}} \to \D{\mod{\ab}_{\KK'}}$ preserves $\mc{C}$-indexed colimits, then so does the forgetful functor $\psi: \calg{\cpx{\ind{\mb{T}}}\scr{\otimes}{}} \brk{\mf{W}^{-1}} \to \D{\ind{\mb{T}}}$.
\end{lemma}

\begin{proof}
Suppose $\psi'$ preserves $\mc{C}$-indexed colimits.
Let $\overline{\omega}^{\otimes} := \cpx{\ind{\omega}}\scr{\otimes}{}$ and let $\calg{\overline{\omega}^{\otimes}}: \calg{\cpx{\ind{\mb{T}}}\scr{\otimes}{}} \brk{\mf{W}^{-1}} \to \calg{\cpx{\mod{\ab}_{\KK'}}\scr{\otimes}{}} \brk{\mf{W}'^{-1}}$ denote the induced functor.
It sends $\mf{W}$ to $\mf{W}'$ by \ref{tannakian.6}.
We claim that there is a homotopy equivalence $\psi' \calg{\overline{\omega}^{\otimes}} \simeq \overline{\omega} \psi$.
Indeed, the corresponding square of model categories is essentially commutative by inspection, and each functor involved preserves weak equivalences.
Passing to underlying qcategories, we obtain the desired homotopy commutative square.

Let $\gamma \mapsto A_{\gamma}: \mc{C} \to \calg{\cpx{\ind{\mb{T}}}\scr{\otimes}{}}$ be a functor.
We claim that the canonical morphism $\colim_{\gamma \in \mc{C}} \psi A_{\gamma} \to \psi \colim_{\gamma \in \mc{C}} A_{\gamma}$ is an equivalence.
As $\overline{\omega}$ reflects weak equivalences by \ref{tannakian.6}, it suffices to show that
$
\overline{\omega} \colim_{\gamma \in \mc{C}} \psi A_{\gamma}
\simeq
\overline{\omega} \psi \colim_{\gamma \in \mc{C}} A_{\gamma}
$
is an equivalence.
Note that $\overline{\omega}^{\otimes}$ is $\mf{U}$-cocontinuous.
In particular, it admits a lax symmetric monoidal right adjoint $\upsilon^{\otimes}$ (\cite[7.3.2.7]{Lurie_higher-algebra}), and $\calg{\upsilon^{\otimes}}$ is right adjoint to $\calg{\overline{\omega}^{\otimes}}$, which is thus $\mf{U}$-cocontinuous. 
These remarks, along with the hypothesis that $\psi'$ preserve $\mc{C}$-indexed colimits, provide a homotopy commutative diagram
\[
\xymatrix{
\colim_{\gamma \in \mc{C}} \psi' \calg{\overline{\omega}^{\otimes}} A_{\gamma}
\ar[r]_-{\sim}
\ar[d]^-{\sim}
&
\colim_{\gamma \in \mc{C}} \overline{\omega} \psi A_{\gamma}
\ar[r]_-{\sim}
&
\overline{\omega} \colim_{\gamma \in \mc{C}} \psi A_{\gamma}
\ar[d]
\\
\psi' \colim_{\gamma \in \mc{C}} \calg{\overline{\omega}^{\otimes}} A_{\gamma}
\ar[r]^-{\sim}
&
\psi' \calg{\overline{\omega}^{\otimes}} \colim_{\gamma \in \mc{C}} A_{\gamma}
\ar[r]^-{\sim}
&
\overline{\omega} \psi \colim_{\gamma \in \mc{C}} A_{\gamma}
}
\]
and the claim follows.
\end{proof}

\begin{thm}
\label{tannakian.10}
Let $\mb{T}^{\otimes}$ be a $\mf{U}$-small Tannakian category.
If $\mf{W}$ denotes the class of morphisms of $\calg{\cpx{\ind{\mb{T}}}\scr{\otimes}{}}$ inducing quasi-isomorphisms between the underlying objects of $\cpx{\ind{\mb{T}}}$, then the canonical functor 
$
\phi:
\calg{\cpx{\ind{\mb{T}}}\scr{\otimes}{}} \brk{\mf{W}^{-1}} 
\to 
\calg{\D{\ind{\mb{T}}}\scr{\otimes}{}}
$
is an equivalence.
\end{thm}

\begin{proof}
As in \ref{tannakian.5}, choose a $\KK$-linear, faithful, exact symmetric monoidal functor $\omega^{\otimes}: \mb{T}^{\otimes} \to \mod{\ab}_{\KK'}\scr{\otimes}{\aleph_0}$ for some field extension $\hom{\1{\mb{T}}}{\1{\mb{T}}}_{\mb{T}} \hookrightarrow \KK'$ and set $\overline{\omega}^{\otimes} := \cpx{\ind{\omega}}\scr{\otimes}{}$.
By \cite[7.1.4.7, 4.5.4.7]{Lurie_higher-algebra}, $\calg{\cpx{\mod{\ab}_{\KK'}\scr{\otimes}{}}} \brk{\mf{W}'^{-1}} \to \calg{\D{\mod{\ab}_{\KK'}}\scr{\otimes}{}}$ is an equivalence, where $\mf{W}'$ denotes the class of morphisms inducing quasi-isomorphisms between the underlying complexes of $\KK'$-modules.
To prove the claim, we use this special case and the properties of $\overline{\omega}^{\otimes}$ to show that the conditions of \cite[4.7.4.16]{Lurie_higher-algebra} are satisfied.
Consider the diagram
\[
\xymatrix{
\calg{\cpx{\ind{\mb{T}}}\scr{\otimes}{}} \brk{\mf{W}^{-1}}
\ar[rr]_-{\phi}
\ar[dr]_-G
&
&
\calg{\D{\ind{\mb{T}}}\scr{\otimes}{}}
\ar[dl]^-{G'}
\\
&
\D{\ind{\mb{T}}}
}
\]
\begin{enumerate}[topsep=0ex, itemsep=0ex, label=(\arabic*), itemindent=\parindent]
\item
The qcategories $\D{\ind{\mb{T}}}$ and $\calg{\cpx{\ind{\mb{T}}}\scr{\otimes}{}} \brk{\mf{W}^{-1}}$ are locally $\mf{U}$-presentable by \ref{tannakian.1}$(ii)$ and \ref{tannakian.8}: the qcategory underlying a $\mf{U}$-combinatorial model category is locally $\mf{U}$-presentable (\cite[1.3.4.22]{Lurie_higher-algebra}).
The qcategory $\calg{\D{\ind{\mb{T}}} \scr{\otimes}{} }$ is also locally $\mf{U}$-presentable by \cite[4.1.4.8, 3.2.3.5]{Lurie_higher-algebra} and \ref{tannakian.7}.
The forgetful functor 
\[
\calg{\cpx{\ind{\mb{T}}}\scr{\otimes}{}} 
\to 
\cpx{\ind{\mb{T}}}
\]
admits a left adjoint given by the free commutative algebra functor, denoted by $\sym{}$, and these form a Quillen adjunction by definition of the model structure on $\calg{\cpx{\ind{\mb{T}}}\scr{\otimes}{}}$.
In particular, by \cite[1.3.4.27]{Lurie_higher-algebra}, as $G$ is obtained from a right Quillen functor by passing to underlying qcategories, it admits a left adjoint $F$.
By \cite[3.1.3.5]{Lurie_higher-algebra}, $G'$ also admits a left adjoint $F'$.
\item
The functor $G$  preserves geometric realizations of simplicial objects by \ref{tannakian.9} and \cite[7.1.4.7, 4.5.4.12]{Lurie_higher-algebra}.
The functor $G'$ preserves geometric realizations of simplicial objects by \cite[3.2.3.2]{Lurie_higher-algebra}.
The functor $G$ is conservative as it tautologically preserves weak equivalences, and $G'$ is conservative by \cite[3.2.2.6]{Lurie_higher-algebra}.
\item
We now claim that, for each $K \in \cpx{\ind{\mb{T}}}$, the canonical morphism $G'F'\prns{K} \to GF\prns{K}$ is an equivalence. 
As explained in step (e) of the proof of \cite[4.5.4.7]{Lurie_higher-algebra}, it suffices to prove that, for each $K$, the colimit defining the total symmetric power $\sym{K} := \coprod_{n \in \mb{Z}_{\geq0}} \on{sym}^n\prns{K}$ is a homotopy colimit.
Let $\on{\mb{L}\mr{sym}}$ denote the corresponding homotopy colimit functor.
As $\overline{\omega}^{\otimes}$ is symmetric monoidal, $\mf{U}$-cocontinuous and homotopically $\mf{U}$-cocontinuous, we have a homotopy commutative square
\[
\xymatrix{
\on{\mb{L}\mr{sym}}\prns{\overline{\omega} K}
\ar[r]_-{\sim}
\ar[d]
&
\overline{\omega}\on{\mb{L}\mr{sym}}\prns{K}
\ar[d]
\\
\sym{\overline{\omega}K}
\ar[r]^-{\sim}
&
\overline{\omega} \sym{K}
}
\]
and the left vertical arrow is an equivalence by \cite[7.1.4.7]{Lurie_higher-algebra} and step $(e)$ of the proof of \cite[4.5.4.7]{Lurie_higher-topos} applied to $\mb{A}^{\otimes} := \cpx{\mod{\ab}_{\KK'}}\scr{\otimes}{}$.
Since $\overline{\omega}$ is conservative (\ref{tannakian.6}), the claim follows.
Thus, the conditions of \cite[4.7.4.16]{Lurie_higher-algebra} are satisfied.
\qedhere
\end{enumerate}
\end{proof}

\section{Mixed Hodge coefficients}

\setcounter{thm}{-1}

\begin{notation}
\label{hodge.0}
Throughout this section, we fix $\KK \hookrightarrow \mb{R}$, a subfield of the real numbers.
\end{notation}

\begin{motivation}
While mixed Hodge structures arise very naturally in algebraic geometry, they tend to do so as the cohomology of much larger objects, to wit, mixed Hodge complexes.
There is thus a dichotomy between complexes of mixed Hodge structures, which are hard to construct but form a very well-behaved category, and mixed Hodge complexes, which are much easier to construct but, as a 1-category, leave much to be desired.
By a result of A.~Beilinson (\cite[3.11]{Beilinson_absolute-hodge}), their derived categories are nevertheless equivalent.
In this section, we translate this into an equivalence of symmetric monoidal qcategories.
\end{motivation}

\begin{summary}
We begin by reviewing the construction of the symmetric monoidal differential graded category of mixed Hodge complexes (\ref{hodge.1}, \ref{hodge.2}, \ref{hodge.3}).
Theorems \ref{hodge.6} and \ref{hodge.7} lift A.~Beilinson's equivalence (\cite[3.11]{Beilinson_absolute-hodge}) to an equivalence of symmetric monoidal qcategories.
\end{summary}

\begin{defn}
[{\cite[3.9]{Beilinson_absolute-hodge}}]
\label{hodge.1}
We recall the following constructions:
\begin{enumerate}
\item
A \emph{mixed Hodge $\KK$-complex} $\prns{K_{\cdot}, F, W, \alpha, \beta}$ is a diagram 
\begin{equation}
\label{hodge.1.1}
\prns{K_1, W}
\xrightarrow{\alpha}
\prns{K_2, W}
\xleftarrow{\beta}
\prns{K_3, F, W}
\end{equation}
in which $\prns{K_1, W}$ \resp{$\prns{K_2, W}$, resp. $\prns{K_3, F, W}$} is a filtered \resp{filtered, resp. bifiltered} complex of $\KK$-modules \resp{$\mb{C}$-modules, resp. $\mb{C}$-modules}, $\alpha$ is a filtered quasi-isomorphism $\prns{K_1 \otimes_{\KK} \mb{C}, W \otimes_{\KK} \mb{C}} \isom \prns{K_2, W}$, $\beta$ is a filtered quasi-isomorphism $\prns{K_3,W} \isom \prns{K_2, W}$ and the following conditions are satisfied:
\begin{enumerate}
\item
$\bigoplus_{n \in \mb{Z}}\h^n K_1$ is of finite rank; 
\item
for each $k \in \mb{Z}$, the differentials of the complex $\gr{K_3}^W_k$ are strictly compatible with the filtration induced by $F$; and
\item
for each $\prns{k, n} \in \mb{Z}^2$, the isomorphism $\beta^{-1}\alpha: \h^n\gr{K_1}^W_k \otimes_{\KK} \mb{C} \isom \h^n \gr{K_3}^W_k$ and the filtration induced by $F$ endow $\h^n \gr{K_1}^W_k$ with a pure Hodge structure of weight $k+n$.
\end{enumerate}
We refer to the filtrations denoted by ``$W$'' as \emph{weight} filtrations and those by ``$F$'' as \emph{Hodge} filtrations.
By convention, $W$ will always be increasing and $F$ will always be decreasing.
We will also frequently suppress the morphisms $\alpha$ and $\beta$ and refer abusively to the mixed Hodge $\KK$-complex $\prns{K_{\cdot}, F, W}$.
Note that this definition is not equivalent to \cite[3.2]{Beilinson_absolute-hodge}, but rather to \cite[3.9]{Beilinson_absolute-hodge}, the difference being the shift in the weights by the cohomological degree appearing in condition $(\textup{c})$ above.
It seems likely that the following techniques apply to both settings with slight modification.
\item
We say that a mixed Hodge $\KK$-complex $\prns{K_{\cdot}, F, W, \alpha, \beta}$ is \emph{polarizable} if, for each $\prns{k, n} \in \mb{Z}^2$, the pure Hodge $\KK$-structure $\prns{\h^n \gr{K_1}^W_k, F}$ is polarizable.
\item
We define a \emph{morphism of polarizable mixed Hodge $\KK$-complexes} $f: \prns{K_{\cdot}, F, W, \alpha, \beta} \to \prns{K'_{\cdot}, F, W, \alpha', \beta'}$ to be a morphism of diagrams, consisting of morphisms of (bi)filtered complexes $f_1: \prns{K_1, W} \to \prns{K'_1, W}$, $f_2: \prns{K_2, W} \to \prns{K'_2, W}$ and $f_3: \prns{K_3, F, W} \to \prns{K'_3, F, W}$ such that $\alpha' \prns{f_1\otimes_{\KK} \mb{C}} = \prns{f_2 \otimes_{\KK} \mb{C}} \alpha$ and $\beta' f_3 = f_3 \beta$.
Let $\MHC^{\mr{p}}_{\KK}$ denote the category polarizable mixed Hodge $\KK$-complexes and morphisms of such.

The category $\MHC^{\mr{p}}_{\KK}$ inherits a $\KK$-linear-differential-graded-category structure from the differential graded categories of filtered \resp{filtered, resp. bifiltered} complexes of $\KK$-modules \resp{$\mb{C}$-modules, resp. $\mb{C}$-modules} as explained in \cite[1.2.1]{Ivorra_mixed-hodge}.

With translations and cones defined in the evident way (\cite[1.2.6]{Ivorra_mixed-hodge}), $\MHC^{\mr{p}}_{\KK}$ is pretriangulated.
Note that $\MHC^{\mr{p}}_{\KK}$ is locally $\mf{U}$-small, but neither essentially $\mf{U}$-small nor locally $\mf{U}$-presentable, so we will view it as a $\mf{V}$-small category for some suitably large Grothendieck universe $\mf{V}$ containing $\mf{U}$.
\end{enumerate}
\end{defn}

\begin{defn}
\label{hodge.2}
Let $K$ and $L$ be two complexes of $\KK$-modules.
\begin{enumerate}
\item
In \eqref{tannakian.1.1}, we defined the tensor product $K \otimes L$.
If $K$ and $L$ are given filtrations $F$ and $G$, respectively, then we define the tensor product $\prns{K, F} \otimes_{\KK} \prns{L, G}$ to be $K \otimes_{\KK} L$ equipped with the filtration given by
\[
\prns{F \otimes_{\KK} G}^k\prns{K \otimes_{\KK} L}^n
:=
\bigoplus_{r\in \mb{Z}} \sum_{s \in \mb{Z}} \prns{F^{s}K^{r} \otimes_{\KK} G^{k-s}L^{n-r}}
\]
for each $\prns{k, n} \in \mb{Z}^2$.
Defining the tensor product of the additional filtrations analogously, we have a reasonable construction of the tensor product of bifiltered complexes.

With this definition, the category of increasingly filtered complexes of $\KK$-modules is a symmetric monoidal category with unit given by $\KK$ as a complex concentrated in degree $0$ equipped with the trivial filtration $F^k\KK :=0$ for $k \in \mb{Z}_{<0}$ and $F^k\KK = \KK$ for $k \in \mb{Z}_{\geq0}$.
This too extends easily to the setting of bifiltered complexes.
\item
The \emph{tensor product} $\prns{K_{\cdot} \otimes K'_{\cdot}, F \otimes F', W \otimes W', \alpha \otimes \alpha', \beta \otimes \beta'}$ of two mixed Hodge $\KK$-complexes $\prns{K_{\cdot}, F, W, \alpha, \beta}$ and $\prns{K'_{\cdot}, F', W', \alpha', \beta'}$
is defined by
\[
\prns{K_1 \otimes_{\KK} K'_1, W \otimes_{\KK} W'} 
\xrightarrow{\alpha \otimes_{\mb{C}} \alpha'}
\prns{K_2 \otimes_{\mb{C}} K'_2, W \otimes_{\mb{C}} W'}
\xleftarrow{\beta \otimes_{\mb{C}} \beta'}
\prns{K_3 \otimes_{\mb{C}} K'_3, F \otimes_{\mb{C}} F', W \otimes_{\mb{C}} W'}.
\]
A filtered variant of the K\"unneth formula, along with the observation that the tensor product of two polarizable pure Hodge $\KK$-structures is another such, show that this is another object of $\MHC^{\mr{p}}_{\KK}$ (\cite[3.20]{Peters-Steenbrink_mixed-hodge}), so this tensor product makes $\MHC^{\mr{p}}_{\KK}$ a symmetric monoidal category $\MHC^{\mr{p},\otimes}_{\KK}$.

The tensor product bifunctor is compatible with the differential graded structures on filtered and bifiltered complexes, so $\MHC^{\mr{p}, \otimes}_{\KK}$ is in fact a symmetric monoidal $\KK$-linear differential graded category, i.e., a commutative monoid in the symmetric monoidal category $\DGCAT_{\KK}^{\otimes}$ of $\mf{V}$-small $\KK$-linear differential graded categories (\cite[2.1$(vi)$]{Drew_verdier-quotients}).
\item
From the pretriangulated $\KK$-linear symmetric monoidal differential graded category $\MHC^{\mr{p}, \otimes}_{\KK}$, we construct a $\mf{V}$-small stable symmetric monoidal qcategory $\altdgnerve{\MHC^{\mr{p}}_{\KK}}\scr{\otimes}{} \in \calg{\QCAT^{\mr{Ex}, \otimes}}$, using \cite[2.5]{Drew_verdier-quotients}, where $\QCAT^{\mr{Ex}, \otimes}$ denotes the qcategory of $\mf{V}$-small stable qcategories  equipped with the symmetric monoidal structure of \cite[5.4.7]{Lurie_DAGVIII}.
The homotopy category $\ho{\altdgnerve{\MHC^{\mr{p}}_{\KK}}}$ is equivalent to $\mr{H}^0\prns{\MHC^{\mr{p}}_{\KK}}$ (\cite[2.1$(viii)$]{Drew_verdier-quotients}).
\end{enumerate}
\end{defn}

\begin{defn}
\label{hodge.3}
To each $K \in \cpx{\mhs^{\mr{p}}_{\KK}}^{\mr{b}}$ we assign a diagram $\prns{K_{\cdot}, F, W}$ as in \eqref{hodge.1.1} satisfying conditions $\textup{(a)}$ and $\textup{(b)}$ of \ref{hodge.1}$(i)$ in the evident way: $K_1$ and $K_2 = K_3$ are the complexes of $\KK$-modules and $\mb{C}$-modules underlying $K$, respectively, and $F$ and $W$ are the Hodge and weight filtrations, respectively.
We must, however, shift the weight filtration in order to obtain a diagram satisfying \ref{hodge.1}$(i)\textup{(c)}$, setting $\tilde{W}_k\prns{K^n} := W_{k+n}\prns{K^n}$ for each $\prns{k,n} \in \mb{Z}^2$.
This assignment $K \mapsto \prns{K_{\cdot}, F, \tilde{W}}$ extends to a $\KK$-linear differential graded functor $\chi:  \cpx{\mhs^{\mr{p}}_{\KK}}^{\mr{b}} \to \MHC^{\mr{p}}_{\KK}$.
\end{defn}

\begin{prop}
\label{hodge.4}
The $\KK$-linear differential graded functor $\chi: \cpx{\mhs^{\mr{p}}_{\KK}}^{\mr{b}} \to \MHC^{\mr{p}}_{\KK}$ is symmetric monoidal with respect to the symmetric monoidal structures of \textup{\ref{tannakian.1}$(iii)$} and \textup{\ref{hodge.2}}$(ii)$ and induces an exact symmetric monoidal functor $\chi^{\otimes}: \mc{K}^{\mr{b}}\prns{\mhs^{\mr{p}}_{\KK}}\scr{\otimes}{} \to \altdgnerve{\MHC^{\mr{p}}_{\KK}}\scr{\otimes}{}$.
\end{prop}

\begin{proof}
Once we have shown that $\chi$ underlies a $\KK$-linear symmetric monoidal differential graded functor, the last assertion will follow from \cite[2.5]{Drew_verdier-quotients}.
Let $K$ and $L$ be objects of $\cpx{\mhs^{\mr{p}}_{\KK}}^{\mr{b}}$.
By definition, then complex of $\KK$-modules underlying $\chi \prns{K \otimes L}$ is the complex of $\KK$-modules underlying $K \otimes L$.
To see that $\chi$ is symmetric monoidal, it therefore suffices to check that the Hodge and weight filtrations on $\chi\prns{K \otimes L}$ are equal to the tensor products of the Hodge and weight filtrations, respectively, on $\chi\prns{K}$ and $\chi\prns{L}$.
Fix $\prns{k, n} \in \mb{Z}^2$.
We have the following computations:
\begin{align*}
\tilde{W}_k\prns{\chi\prns{K \otimes_{\KK} L}}^n
=
W_{k+n} \prns[2]{\bigoplus_{r \in \mb{Z}} \prns{K^r \otimes_{\KK} L^{n-r}}}
&=
\bigoplus_{r \in \mb{Z}} \sum_{s \in \mb{Z}} \prns{W_sK^r \otimes_{\KK} W_{k+n -s}L^{n-r}},
\\
\prns{\tilde{W} \otimes_{\KK} \tilde{W}}^k \prns{\chi\prns{K} \otimes_{\KK} \chi \prns{L}}^n
&=
\bigoplus_{r \in \mb{Z}} \sum_{t \in \mb{Z}} \prns{W_{r+t} K^r \otimes_{\mb{K}} W_{k+n-\prns{r+t}} L^{n-r}}.
\end{align*}
Reindexing the last expression by $t := s-r$, we find that the two filtrations are equal.
Essentially the same argument applies for the Hodge filtrations.
Sparing the reader the predictably tedious verification that $\chi^{\otimes}$ satisfies the coherence properties required of a symmetric monoidal differential graded functor, the claim follows.
\end{proof}

\begin{defn}
\label{hodge.5}
Let $\prns{K_{\cdot}, F, W} \in \MHC^{\mr{p}}_{\KK}$.
We say that $\prns{K_{\cdot}, F, W}$ is \emph{acyclic} if the underlying complex of $\KK$-modules $K_1$ is acyclic.
Let $\iota: \on{\mc{A}c} \hookrightarrow \altdgnerve{\MHC^{\mr{p}}_{\KK}}$ denote the full subqcategory spanned by the acyclic objects.
Since the forgetful differential graded functor $\prns{K_{\cdot}, F, W} \mapsto K_1: \MHC^{\mr{p}}_{\KK} \to \cpx{\mod{\ab}_{\KK}}$ preserves cones and translations, $\on{\mc{A}c}$ is a stable subqcategory and $\iota$ is an exact functor.
We define $\mhc^{\mr{p}}_{\KK}:= \altdgnerve{\MHC^{\mr{p}}_{\KK}}/\on{\mc{A}c}$ to be the cofiber of $\iota$ in $\QCAT^{\mr{Ex}}$, i.e., the Verdier quotient of $\altdgnerve{\MHC^{\mr{p}}_{\KK}}$ by $\on{\mc{A}c}$.
\end{defn}

\begin{thm}
[{\cite[3.11]{Beilinson_absolute-hodge}}]
\label{hodge.6}
The exact functor $\chi: \K{\mhs^{\mr{p}}_{\KK}}^{\mr{b}} \to \altdgnerve{\MHC^{\mr{p}}_{\KK}}$ of \textup{\ref{hodge.4}} induces an equivalence $\overline{\chi}: \D{\mhs^{\mr{p}}_{\KK}}^{\mr{b}} \isom \mhc^{\mr{p}}_{\KK}$.
\end{thm}

\begin{proof}
A complex of polarizable mixed Hodge $\KK$-structures is acyclic if and only if the complex of underlying $\KK$-modules is acyclic, so $\chi$ induces an exact functor $\overline{\chi}: \D{\mhs^{\mr{p}}_{\KK}}^{\mr{b}} \to \mhc^{\mr{p}}_{\KK}$ by the universal property of the cofiber defining $\D{\mhs^{\mr{p}}_{\KK}}$ (\ref{tannakian.1}$(i)$).
At the level of homotopy categories, $\overline{\chi}$ induces the functor of \cite[3.11]{Beilinson_absolute-hodge}, so $\ho{\overline{\chi}}$ is an equivalence.
This proves the first assertion: the exact functor $\overline{\chi}$, whose domain and codomain are stable qcategories, is an equivalence if and only if $\ho{\overline{\chi}}$ is an equivalence.
\end{proof}

\begin{thm}
\label{hodge.7}
The canonical functor $\pi: \altdgnerve{\MHC^{\mr{p}}_{\KK}} \to \mhc^{\mr{p}}_{\KK}$ underlies a symmetric monoidal Verdier quotient $\pi^{\otimes}: \altdgnerve{\MHC^{\mr{p}}_{\KK}}\scr{\otimes}{} \to \mhc^{\mr{p}, \otimes}_{\KK}$ of $\altdgnerve{\MHC^{\mr{p}}_{\KK}}\scr{\otimes}{}$ by $\on{\mc{A}c}$ \textup{\ref{tannakian.1}$(iv)$} and the equivalence $\overline{\chi}$ of \textup{\ref{hodge.6}} underlies a symmetric monoidal equivalence $\overline{\chi}^{\otimes}: \D{\mhs^{\mr{p}}_{\KK}}^{\mr{b}}\scr{\otimes}{} \to \mhc^{\mr{p}, \otimes}_{\KK}$.
\end{thm}

\begin{proof}
The functor $\prns{K_{\cdot}, F, W} \mapsto K_1: \MHC^{\mr{p}, \otimes}_{\KK} \to \cpx{\mod{\ab}_{\KK}}$ is symmetric monoidal and reflects acyclicity, so it follows from \cite[3.4]{Drew_verdier-quotients} applied to $\mb{A}^{\otimes} = \mod{\ab}_{\KK}\scr{\otimes}{}$ that the tensor product in $\MHC^{\mr{p}}_{\KK}$ preserves acyclic objects separately in each variable.
The first assertion follows by \cite[2.6]{Drew_verdier-quotients}.

By the universal property of the symmetric monoidal Verdier quotient $\pi^{\otimes}$, the composite $\pi^{\otimes}\chi^{\otimes}$ factors through the symmetric monoidal Verdier quotient $\K{\mhs^{\mr{p}}_{\KK}}^{\mr{b}}\scr{\otimes}{} \to \D{\mhs^{\mr{p}}_{\KK}}^{\mr{b}}\scr{\otimes}{}$ of \ref{tannakian.1}$(iv)$.
The exact functor underlying the resulting symmetric monoidal functor $\overline{\chi}^{\otimes}: \D{\mhs^{\mr{p}}_{\KK}}^{\mr{b}}\scr{\otimes}{} \to \mhc^{\mr{p}, \otimes}_{\KK}$ must be equivalent to $\chi$ by the universal property of the cofiber $\D{\mhs^{\mr{p}}_{\KK}} := \K{\mhs^{\mr{p}}_{\KK}}^{\mr{b}}/\on{\mc{A}c}\prns{\mhs^{\mr{p}}_{\KK}}$, where $\on{\mc{A}c}\prns{\mhs^{\mr{p}}_{\KK}} \hookrightarrow \K{\mhs^{\mr{p}}_{\KK}}^{\mr{b}}$ denotes the full subqcategory spanned by acyclic complexes.
By \cite[2.1.3.8]{Lurie_higher-algebra}, the fact (\ref{hodge.6}) that $\overline{\chi}$ is an equivalence implies that the symmetric monoidal functor $\overline{\chi}^{\otimes}$ is an equivalence.
\end{proof}

\section{Rectification}

\setcounter{thm}{-1}

\begin{notation}
\label{rect.0}
Throughout this section, we fix the following notation:
\begin{enumerate}
\item
$\KK \hookrightarrow \mb{R}$, a subfield of the real numbers; and
\item
$\kk \hookrightarrow \mb{C}$, a subfield of the complex numbers.
\end{enumerate}
\end{notation}

\begin{motivation}
We arrive now at our intended applications.
Having in the previous two sections constructed the requisite equivalence between the symmetric monoidal qcategories of complexes of mixed Hodge structures and mixed Hodge complexes and the necessary ingredients for the rectification of presheaves of commutative algebras in the symmetric monoidal derived category of mixed Hodge structures, we now perform the desired rectifications.
Specifically, we show that the functor assigning to each $X \in \sch[ft]{\kk}$ the graded polarizable mixed Hodge structure $\mr{H}^{\bullet}_{\mr{Betti}}\prns{X, \KK}$ of \cite[8.2.1]{Deligne_hodgeIII}, equipped with the ring structure given by the cup product, can be obtained by taking the cohomology of a presheaf 
\[
\tilde{\Gamma}_{\mr{Hdg}}:
\prns{\sch[ft]{\kk}}\op
\to
\calg{\cpx{\ind{\mhs^{\mr{p}}_{\KK}}}\scr{\otimes}{}}
\]
of commutative algebras in the symmetric monoidal category $\cpx{\ind{\mhs^{\mr{p}}_{\KK}}}\scr{\otimes}{}$.
\end{motivation}

\begin{summary}
We begin by constructing the presheaf $\tilde{\Gamma}_{\mr{Hdg}}$ on the category of separated $\kk$-schemes of finite type (\ref{rect.2}).
In order to extend $\tilde{\Gamma}_{\mr{Hdg}}$ to singular $\kk$-schemes, we appeal to a result of V.~Voevodsky that requires some terminology from $\mb{A}^1$-homotopy theory, which we recall in \ref{rect.4}.
Proposition \ref{rect.5} is a general result providing sufficient conditions for a presheaf on $\sm[sft]{\kk}$ to extend naturally to a functor on $\sch[ft]{\kk}$.
The desired presheaf on $\sch[ft]{\kk}$ is then constructed in \ref{rect.6}.
\end{summary}

\begin{defn}
\label{rect.1}
We denote by $\on{\mc{C}pt}$ the category of \emph{smooth compactifications}, whose objects are the dense open immersions $j: X \hookrightarrow \overline{X}$ in $\sm[sft]{\mb{C}}$ such that $\overline{X}$ is smooth and proper over $\spec{\mb{C}}$ and $\overline{X} - X$ is a normal crossings divisor, and whose morphisms are commutative squares in $\sm[sft]{\mb{C}}$.
We abusively denote objects of $\on{\mc{C}pt}$ by ordered pairs $\prns{X, \overline{X}}$, suppressing the morphism $j$.
\end{defn}

\begin{thm}
\label{rect.2}
There exists a functor $\tilde{\Gamma}_{\on{Hdg}}: \prns{\sm[sft]{\kk}}\op \to \calg{\cpx{\ind{\mhs^{\mr{p}}_{\KK}}}\scr{\otimes}{}}$ such that, for each $X \in \sm[sft]{\kk}$ and each $r \in \mb{Z}$, $\h^r\tilde{\Gamma}_{\on{Hdg}}\prns{X}$ is naturally isomorphic to the $\KK$-linear Betti cohomology $\on{H}^r_{\on{Betti}}\prns{X\otimes_{\kk}\mb{C}, \KK}$ equipped with the mixed Hodge $\KK$-structure of \textup{\cite[3.2.5]{Deligne_hodgeII}}.
\end{thm}

\begin{proof}
It suffices to treat the case in which $\kk = \mb{C}$ and then compose with the functor $\prns{\sm[sft]{\kk}}\op \to \prns{\sm[sft]{\mb{C}}}\op$.
Making the obvious modification from the $\mb{Q}$-linear to the $\KK$-linear setting and forgetting the $\mb{Z}$-linear component, \cite[8.15]{Navarro-Aznar_hodge-deligne} provides us with a functor $\overline{\Gamma}_0: \on{\mc{C}pt}\op \to \calg{\mr{Z}^0\prns{\MHC^{\mr{p}, \otimes}_{\KK}}}$.
By construction, for each $r \in \mb{Z}$, the image of an object $\prns{X, \overline{X}}$ under the composite
\begin{equation}
\label{rect.2.1}
\on{\mc{C}pt}\op
\xrightarrow{\overline{\Gamma}_0}
\calg{\mr{Z}^0\prns{\MHC^{\mr{p}, \otimes}_{\KK}}}
\to
\mr{Z}^0\prns{\MHC^{\mr{p}}_{\KK}}
\to
\mhc^{\mr{p}}_{\KK}
\xrightarrow{\chi^{-1}}
\D{\mhs^{\mr{p}}_{\KK}}^{\mr{b}}
\xrightarrow{\h^r}
\mhs^{\mr{p}}_{\KK}
\end{equation}
is naturally isomorphic to Deligne's mixed Hodge structure on $\on{H}^r_{\on{Betti}}\prns{X, \KK}$ as constructed in \cite[3.2.5$(iii)$]{Deligne_hodgeII}.

Let $\mf{W}$ be the class of weak equivalences of the model structure of \ref{tannakian.8} on the category $\calg{\cpx{\ind{\mhs^{\mr{p}}_{\KK}}}\scr{\otimes}{}}$.
We construct $\overline{\Gamma}: \on{\mc{C}pt}\op \to \calg{\cpx{\ind{\mhs^{\mr{p}}_{\KK}}}\scr{\otimes}{}}\brk{\mf{W}^{-1}}$ as the following composite:
\[
\xymatrix{
\on{\mc{C}pt}\op
\ar[dd]_-{\overline{\Gamma}}
\ar[r]_-{\overline{\Gamma}_0}
&
\calg{\mr{Z}^0\prns{\MHC^{\mr{p}, \otimes}_{\KK}}}
\ar[r]_-{q}
&
\calg{\altdgnerve{\MHC^{\mr{p}}_{\KK}}\scr{\otimes}{}}
\ar[d]^-{\pi}
\\
&
&
\calg{\mhc^{\mr{p}, \otimes}_{\KK}}
\ar[d]^-{\overline{\chi}^{-1}}_-{\sim}
\\
\calg{\cpx{\ind{\mhs^{\mr{p}}_{\KK}}}\scr{\otimes}{}}\brk{\mf{W}^{-1}}
&
\calg{\D{\ind{\mhs^{\mr{p}}_{\KK}}}\scr{\otimes}{}}
\ar[l]_-{\phi^{-1}}^-{\sim}
&
\calg{\D{\mhs^{\mr{p}}_{\KK}}^{\mr{b}}\scr{\otimes}{}}.
\ar@{_{(}->}[l]_-{\iota}
}
\]
Here, $q$ is the functor given by \cite[2.7]{Drew_verdier-quotients}, $\pi$ and $\overline{\chi}^{-1}$ are those of \ref{hodge.7}, $\iota$ is that of \cite[4.7]{Drew_verdier-quotients} and $\phi^{-1}$ is that of \ref{tannakian.10}. 
Note that $\iota$ and $\phi^{-1}$ do not affect cohomology objects, so $\overline{\Gamma}$ also recovers Deligne's mixed Hodge structures by \eqref{rect.2.1}.
As 
\[
\overline{\Gamma} \in \fun{\on{\mc{C}pt}\op}{\calg{\cpx{\ind{\mhs^{\mr{p}}_{\KK}}}\scr{\otimes}{}} \brk{\mf{W}^{-1}}}
\]
and its codomain is the qcategory underlying a $\mf{U}$-combinatorial model category (\ref{tannakian.8}), \cite[1.3.4.25]{Lurie_higher-algebra} implies that $\overline{\Gamma}$ can be rectified to $\overline{\Gamma}: \on{\mc{C}pt}\op \to \calg{\cpx{\ind{\mb{\mhs^{\mr{p}}_{\KK}}}}\scr{\otimes}{}}$.

We now have a functor $\overline{\Gamma}$ between $1$-categories.
Let $\on{\mc{C}pt}_X \subseteq \on{\mc{C}pt}$ denote the subcategory of smooth compactifications of a fixed object $X \in \sm[sft]{\mb{C}}$, i.e., the subcategory spanned by the morphisms $\prns{f, \overline{f}}$ such that $f = \id_X$.
Then $\on{\mc{C}pt}_X$ is nonempty by theorems of M.~Nagata (\cite[4.1]{Conrad_deligne's-notes}) and H.~Hironaka (\cite{Hironaka_resolution-of-singularities}) and $\aleph_0$-filtered by a standard argument (\cite[3.2.11]{Deligne_hodgeII}).
Also, if $\prns{f, \overline{f}}: \prns{X, \overline{X}} \to \prns{X, \overline{X}'}$ is a morphism of $\on{\mc{C}pt}$, then $\overline{\Gamma}\prns{f, \overline{f}} \in \mf{W}$.
Indeed, the morphism of complexes of $\KK$-modules underlying $\overline{\Gamma}_0\prns{f, \overline{f}}$ is the identity on the singular cochain complex of $X$.
We may therefore construct the desired functor $\tilde{\Gamma}_{\mr{Hdg}}: \prns{\sm[sft]{\mb{C}}}\op \to \calg{\cpx{\ind{\mhs^{\mr{p}}_{\KK}}}\scr{\otimes}{}}$ by defining
\[
\tilde{\Gamma}_{\mr{Hdg}}\prns{X}
:=
\colim_{\overline{X} \in \on{\mc{C}pt}_X} \overline{\Gamma}\prns{X, \overline{X}}
\]
and having $\tilde{\Gamma}_{\mr{Hdg}}$ act in the evident way on morphisms.
\end{proof}

\begin{cor}
\label{rect.3}
The functor $\tilde{\Gamma}_{\mr{Hdg}}$ of \textup{\ref{rect.2}} induces a functor $\Gamma_{\mr{Hdg}}: \prns{\sm[sft]{\kk}}\op \to \D{\ind{\mhs^{\mr{p}}_{\KK}}}$ underlying a symmetric monoidal functor $\Gamma^{\otimes}_{\mr{Hdg}}: \prns{\sm[sft]{\kk}}\scr{\mr{op}, \amalg}{} \to \D{\ind{\mhs^{\mr{p}}_{\KK}}}\scr{\otimes}{}$.
\end{cor}

\begin{proof}
Recall that $\prns{\sm[sft]{\kk}}\scr{\mr{op}, \amalg}{}$ denotes the coCartesian symmetric monoidal structure \textup{(\cite[\S2.4.3]{Lurie_higher-algebra})}.
Let $\tilde{\Gamma}'_{\mr{Hdg}}$ be the composite of $\tilde{\Gamma}_{\mr{Hdg}}$, the localization 
\[
\lambda: 
\calg{\cpx{\ind{\mhs^{\mr{p}}_{\KK}}}\scr{\otimes}{}} 
\to 
\calg{\cpx{\ind{\mhs^{\mr{p}}_{\KK}}}\scr{\otimes}{}}\brk{\mf{W}^{-1}}
\]
(\ref{tannakian.8}) and the equivalence $\phi$ of \ref{tannakian.10}.
Let $\Gamma_{\mr{Hdg}}$ denote the composite of $\tilde{\Gamma}'_{\mr{Hdg}}$ with the forgetful functor $\psi: \calg{\D{\ind{\mhs^{\mr{p}}_{\KK}}}\scr{\otimes}{}} \to \D{\ind{\mhs^{\mr{p}}_{\KK}}}$.
By \cite[3.2.4.9]{Lurie_higher-algebra}, $\Gamma_{\mr{Hdg}}$ underlies a lax symmetric monoidal functor $\Gamma^{\otimes}_{\mr{Hdg}}$ with respect to the coCartesian symmetric monoidal structure on $\prns{\sm[sft]{\kk}}\op$ and $\Gamma^{\otimes}_{\mr{Hdg}}$ is symmetric monoidal if $\tilde{\Gamma}'_{\mr{Hdg}}$ preserves finite coproducts.
The object of $\D{\ind{\mhs^{\mr{p}}_{\KK}}}$ underlying the coproduct of $\tilde{\Gamma}'_{\mr{Hdg}}\prns{X}$ and $\tilde{\Gamma}'_{\mr{Hdg}}\prns{Y}$ in $\calg{\D{\ind{\mhs^{\mr{p}}_{\KK}}}\scr{\otimes}{}}$, i.e., its image under $\psi$, is the tensor product $\Gamma_{\mr{Hdg}}\prns{X} \otimes \Gamma_{\mr{Hdg}}\prns{Y}$ (\cite[3.2.4.8]{Lurie_higher-algebra}).
By the K\"unneth formula, $\Gamma_{\mr{Hdg}}\prns{X} \otimes \Gamma_{\mr{Hdg}}\prns{Y} \simeq \Gamma_{\mr{Hdg}}\prns{X \times_{\kk} Y}$ and $X \times_{\kk} Y$ is the coproduct of $X$ and $Y$ in $\prns{\sm[sft]{\kk}}\op$.
Indeed, composing with the conservative symmetric monoidal functor $\D{\ind{\mhs^{\mr{p}}_{\KK}}}\scr{\otimes}{} \to \D{\mod{\ab}_{\KK}}\scr{\otimes}{}$ reduces the problem to the K\"unneth formula for Betti cohomology.
As $\psi$ is conservative, the claim follows.
\end{proof}

\begin{defn}
\label{rect.4}
Let $S$ be a quasi-compact quasi-separated scheme.
Let $\mf{S} \subseteq \sch{S}$ be a full subcategory stable under fiber products and containing $\varnothing$, $S$ and $\mb{A}^1_{S}$, $\mc{C}$ a qcategory, $F: \mf{S}\op \to \mc{C}$ a functor, and $\mf{Q}$ a class of Cartesian squares in $\mf{S}$ of the form
\begin{equation}
\label{rect.4.1}
\xymatrix{
Y'
\ar[r]_{f'}
\ar[d]_-{g'}
\ar@{}[dr]|-Q
&
Y
\ar[d]^-g
\\
X'
\ar[r]^-f
&
X.
}
\end{equation}
\begin{enumerate}
\item
We say that $F$ is:
\begin{enumerate}
\item
\emph{excisive with respect to $\mf{Q}$} if $F\prns{\varnothing}$ is a final object of $\mc{C}$ and, for each $Q \in \mf{Q}$ as in \eqref{rect.4.1}, the square $F\prns{Q}$ is Cartesian in $\mc{C}$;
\item
\emph{Nisnevich excisive} if it is excisive with respect to the class $\mf{Q}_{\mr{Nis}}\prns{\mf{S}}$ of squares of the form \eqref{rect.4.1} such that $f$ is an open immersion, $g$ an \'etale morphism and the induced morphism $g^{-1}\prns{X-X'}_{\mr{red}} \to \prns{X-X'}_{\mr{red}}$ is an isomorphism;
\item
\emph{$\mr{cdh}$-excisive} if it is Nisnevich excisive and also excisive with respect to the class $\mf{Q}_{\mr{cdh}}\prns{\mf{S}}$ of squares of the form \eqref{rect.4.1} such that $f$ is a closed immersion, $g$ is proper and the induced morphism $g^{-1}\prns{X-X'} \to X-X'$ is an isomorphism;
\item
\emph{$\mr{scdh}$-excisive} if it is Nisnevich excisive and also excisive with respect to the class $\mf{Q}_{\mr{scdh}}\prns{\mc{S}}$ of squares of the form \eqref{rect.4.1} such that $X$, $X'$, $Y$ and $Y'$ are smooth $S$-schemes, $f$ is a closed immersion and $g$ is the blow-up of $X$ along $X'$;
\item
\emph{$\mb{A}^1$-invariant} if, for each $X \in \mf{S}$, the morphism $F\prns{X} \to F\prns{\mb{A}^1_X}$ induced by the canonical projection is an equivalence in $\mc{C}$.
\end{enumerate}
\item
Let $\mc{C} = \spc{}$ be the qcategory of spaces.
If $\tau = \mr{Nis}$ \resp{$\tau = \mr{cdh}$, resp. $\tau = \mr{scdh}$}, then we let $\sh{\mf{S}}{\spc{}}_{\tau} \subseteq \psh{\mf{S}}{\mc{C}}$ denote the full subqcategory spanned by the Nisnevich-excisive \resp{$\mr{cdh}$-excisive, resp. $\mr{scdh}$-excisive} functors,
and we let $\H{\mf{S}}_{\tau} \subseteq \sh{\mf{S}}{\spc{}}_{\tau}$ denote the full subqcategory spanned by the functors which are moreover $\mb{A}^1$-invariant.
The inclusions $\sh{\mf{S}}{\spc{}}_{\tau} \hookrightarrow \psh{\mf{S}}{\spc{}}$ and $\H{\mf{S}}_{\tau} \hookrightarrow \sh{\mf{S}}{\spc{}}_{\tau}$ are reflective subqcategories with respective left adjoints $\lambda_{\tau}$ and $\lambda_{\mb{A}^1}$.
Indeed, by the Yoneda lemma (\cite[5.1.3.1]{Lurie_higher-topos}), the excision property for $F \in \psh{\mf{S}}{\spc{}}$ is equivalent to requiring that $F$ be $\mf{W}_{\tau}$-local (\cite[5.5.4.1]{Lurie_higher-topos}), where $\mf{W}_{\tau}$ is the class of morphisms of the form
\[
\zeta_Q:
\yon{}{X'} \amalg_{\yon{}{Y'}} \yon{}{Y}
\to
\yon{}{X}
\]
induced by the universal property of the pushout with $Q \in \mf{Q}_{\tau}\prns{\mf{S}}$, where $\yon{}{}: \mf{S} \hookrightarrow \psh{\mf{S}}{\spc{}}$ denotes the Yoneda embedding.
Similarly, $\mb{A}^1$-invariance is equivalent to requiring that $F$ be $\mf{W}_{\mb{A}^1}$-local, where $\mf{W}_{\mb{A}^1}$ is the class of morphisms of the form
$
\yon{}{\mb{A}^1_X}
\to
\yon{}{X}
$ with $X \in \mf{S}$.
The subqcategories $\sh{\mf{S}}{\spc{}}_{\tau}$ and $\H{\mf{S}}_{\tau}$ are therefore reflective by \cite[5.5.4.15]{Lurie_higher-topos}.
\end{enumerate}
\end{defn}

\begin{prop}
\label{rect.5}
Let $\mc{C}^{\otimes}$ be a stable locally $\mf{U}$-presentable symmetric monoidal qcategory and $\tilde{F}: \prns{\sm[sft]{\kk}}\op \to \calg{\mc{C}^{\otimes}}$ a Nisnevich-excisive, $\mb{A}^1$-invariant functor.
\begin{enumerate}
\item
If $\iota: \sm[sft]{\kk} \hookrightarrow \sch[ft]{\kk}$ denotes the inclusion, there exists a $\mr{cdh}$-excisive, $\mb{A}^1$-invariant functor $\tilde{\overline{F}}: \prns{\sch[ft]{\kk}}\op \to \calg{\mc{C}^{\otimes}}$ such that $\tilde{\overline{F}}\iota\op \simeq \tilde{F}$.
\item
If $\tilde{F}$ corresponds to a symmetric monoidal functor $F^{\otimes}: \prns{\sm[sft]{\kk}}\scr{\mr{op}, \amalg}{} \to \mc{C}^{\otimes}$ via \textup{\cite[2.4.3.18]{Lurie_higher-algebra}}, then $\tilde{\overline{F}}$ corresponds to a symmetric monoidal functor $\overline{F}^{\otimes}: \prns{\sch[ft]{\kk}}\scr{\mr{op}, \amalg}{} \to \mc{C}^{\otimes}$.
\end{enumerate}
\end{prop}

\begin{proof}
By \cite[3.2.3.5]{Lurie_higher-algebra}, $\calg{\mc{C}^{\otimes}}$ is locally $\mf{U}$-presentable.
We claim that $\tilde{F}$ is $\mr{scdh}$-excisive.
The forgetful functor $\phi: \calg{\mc{C}^{\otimes}} \to \mc{C}$ reflects limits (\cite[3.2.2.5]{Lurie_higher-algebra}), so it suffices to show that the composite $F := \phi\tilde{F}$ is $\mr{scdh}$-excisive.
For each $Q \in \mf{Q}_{\mr{scdh}}\prns{\sm[ft]{\kk}}$, the square $F\prns{Q}$ is Cartesian in $\mc{C}$ if and only if, for each $C \in \mc{C}$, the square $\map{C}{F\prns{Q}}_{\mc{C}}$ is Cartesian in $\spc{}_*$.
Indeed, this follows from the Yoneda lemma (\cite[5.1.3.1]{Lurie_higher-topos}).
Thus, $F$ is $\mr{scdh}$-excisive if and only if the functor $\map{C}{F\prns{-}}_{\mc{C}}: \prns{\sm[ft]{\kk}}\op \to \spc{}_*$ is $\mr{scdh}$-excisive for each $C \in \mc{C}$.
Since $\mc{C}$ is stable, we have an equivalence $\map{\Sigma^1C}{F\prns{-}}_{\mc{C}} \simeq \Omega^1\map{C}{F\prns{-}}_{\mc{C}}$ for each $C \in \mc{C}$, where $\Sigma^1$ denotes the suspension endofunctor of $\mc{C}$ and $\Omega^1$ denotes the loop functor of $\spc{}_*$ (\cite[1.1.2.6]{Lurie_higher-algebra}).
The claim that $F$, and hence also $\tilde{F}$, is $\mr{scdh}$-excisive therefore follows from \cite[Lemma 5.1]{Blander_local-projective}. 

By \cite[5.1.5.6]{Lurie_higher-topos}, there is a canonical equivalence
\[
\fun{\psh{\sm[sft]{\kk}}{\spc{}}\op}{\calg{\mc{C}^{\otimes}}}^{\mr{cont}}
\isom
\psh{\sm[sft]{\kk}}{\calg{\mc{C}^{\otimes}}}
\]
under which $\tilde{F}$ corresponds to an object in the essential image of the fully faithful functor
\[
\fun{\H{\sm[sft]{\kk}}_{\mr{scdh}}\op}{\calg{\mc{C}^{\otimes}}}^{\mr{cont}}
\hookrightarrow
\fun{\psh{\sm[sft]{\kk}}{\spc{}}\op}{\calg{\mc{C}^{\otimes}}}^{\mr{cont}}
\]
given by composition with the localization $\lambda^{\mr{op}}_{\mb{A}^1}\lambda^{\mr{op}}_{\mr{scdh}}: \psh{\sm[sft]{\kk}}{\spc{}}\op \to \H{\sm[sft]{\kk}}_{\mr{scdh}}\op$ of \ref{rect.4}$(ii)$.
Let $\tilde{F}': \H{\sm[sft]{\kk}}_{\mr{scdh}}\op \to \calg{\mc{C}^{\otimes}}$ denote the corresponding functor.
By \cite[4.7]{Voevodsky_unstable-motivic}, composition with $\iota\op$ induces an equivalence $\iota^*:\H{\sch[ft]{\kk}}_{\mr{cdh}} \isom \H{\sm[sft]{\kk}}_{\mr{scdh}}$.
Let $\tilde{\overline{F}}$ denote the composite 
\[
\prns{\sch[ft]{\kk}}\op
\to
\H{\sch[ft]{\kk}}_{\mr{cdh}}\op
\xrightarrow{\iota^{*, \mr{op}}}
\H{\sm[sft]{\kk}}_{\mr{scdh}}\op
\xrightarrow{\tilde{F}'}
\calg{\mc{C}^{\otimes}}.
\]
By construction, $\tilde{\overline{F}}$ is $\mr{cdh}$-excisive and $\mb{A}^1$-invariant, and $\tilde{\overline{F}}\iota\op \simeq \tilde{F}$, which proves $(i)$.
Note that $\overline{F}$ inherits $\mb{A}^1$-invariance and $\mr{cdh}$-excisiveness from $\tilde{\overline{F}}$.

Suppose the lax symmetric monoidal functor $F^{\otimes}: \prns{\sm[sft]{\kk}}\scr{\mr{op}, \amalg}{} \to \mc{C}^{\otimes}$ associated with $\tilde{F}$ by \cite[2.4.3.18]{Lurie_higher-algebra} is in fact a symmetric monoidal functor.
We must show that the lax symmetric monoidal functor $\overline{F}^{\otimes}$ associated with $\tilde{\overline{F}}$ is also symmetric monoidal.
By \cite[3.2.4.9]{Lurie_higher-algebra}, this amounts to showing that, for all $X$ and $Y$ in $\sch[sft]{\kk}$, the canonical morphism $\alpha_{XY}: \overline{F}\prns{X} \otimes \overline{F}\prns{Y} \to \overline{F}\prns{X \times_{\kk} Y}$ is an equivalence.

Fix objects $X$ and $Y$ of $\sch[ft]{\kk}$.
Since $\tilde{\overline{F}}$ restricts to a functor equivalent to $\tilde{F}$ on $\prns{\sm[sft]{\kk}}\op$ and $F^{\otimes}$ is symmetric monoidal, $\alpha_{XY}$ is an equivalence if $X$ and $Y$ are smooth, separated $\kk$-schemes.
Also, if $X = \varnothing$, then $X \times_{\kk} Y = \varnothing$ and $\overline{F}\prns{X}$, $\overline{F}\prns{X} \otimes \overline{F}\prns{Y}$ and $\overline{F}\prns{X \times_{\kk} Y}$ are zero objects in $\mc{C}$, so we may assume $X$ and $Y$ are nonempty and $X$ is singular.

Suppose $X$ and $Y$ are separated over $\kk$ and $Y$ is smooth over $\kk$.
If $\dim{X} = 0$, then $X_{\mr{red}}$ is also smooth and $X_{\mr{red}} \times_{\kk} Y = \prns{X \times_{\kk} Y}_{\mr{red}}$.
By $\mr{cdh}$-excision, the inclusion $X_{\mr{red}} \hookrightarrow X$ induces an equivalence $\overline{F}\prns{X} \isom \overline{F}\prns{X_{\mr{red}}}$, so $\alpha_{XY}$ is an equivalence, because we have a homotopy commutative square
\[
\xymatrix{
\overline{F}\prns{X} \otimes \overline{F}\prns{Y}
\ar[r]_-{\alpha_{XY}}
\ar[d]_-{\sim}
&
\overline{F}\prns{X \times_{\kk} Y}
\ar[d]^-{\sim}
\\
\overline{F}\prns{X_{\mr{red}}} \otimes \overline{F}\prns{Y}
\ar[r]^-{\alpha_{X_{\mr{red}}Y}}
&
\overline{F}\prns{X_{\mr{red}} \times_{\kk} Y}
}
\]
If $\dim{X} > 0$, suppose $\alpha_{SY}$ is an equivalence for all separated $\kk$-schemes $S$ of dimension $< \dim{X}$.
By \cite{Hironaka_resolution-of-singularities}, there exists an element
\begin{equation}
\label{rect.5.1}
\xymatrix{
\overline{Z}
\ar[r]
\ar[d]
\ar@{}[dr]|Q
&
\overline{X}
\ar[d]
\\
Z
\ar[r]
&
X
}
\end{equation}
of $\mf{Q}_{\mr{cdh}}\prns{\sch[ft]{\kk}}$ such that $\overline{X}$ is a smooth $\kk$-scheme, $\dim{Z} < \dim{X}$ and $\dim{\overline{Z}} < \dim{\overline{X}} = \dim{X}$. 
Tensoring $\overline{F}\prns{Q}$ with $\overline{F}\prns{Y}$, we have a Cartesian square
\[
\xymatrix{
\overline{F}\prns{X} \otimes \overline{F}\prns{Y}
\ar[r]
\ar[d]
&
\overline{F}\prns{\overline{X}} \otimes \overline{F}\prns{Y}
\ar[d]
\\
\overline{F}\prns{Z} \otimes \overline{F}\prns{Y}
\ar[r]
&
\overline{F}\prns{\overline{Z}} \otimes \overline{F}\prns{Y}.
}
\]
Indeed, $\mc{C}^{\otimes}$ is a stable symmetric monoidal qcategory, so the endofunctor $\prns{-} \otimes C$ is exact for each $C \in \mc{C}$ and, in particular, it preserves Cartesian squares.
The morphisms $\alpha_{ZY}$ and $\alpha_{\overline{Z}Y}$ are equivalences by the inductive hypothesis and $\alpha_{\overline{X}Y}$ is an equivalence since $\overline{X}$ and $Y$ are both smooth, separated $\kk$-schemes.
It follows that $\alpha_{XY}$ is an equivalence.

Suppose both $X$ and $Y$ are separated $\kk$-schemes.
If $\dim{X} = 0$, then $X_{\mr{red}}$ is smooth over $\kk$, so $\alpha_{X_{\mr{red}}Y}$ is an equivalence by the previous case and we have a homotopy commutative square
\[
\xymatrix{
\overline{F}\prns{X} \otimes \overline{F}\prns{Y}
\ar[r]
\ar[d]_-{\sim}
&
\overline{F}\prns{X \times_{\kk} Y}
\ar[d]^-{\sim}
\\
\overline{F}\prns{X_{\mr{red}}} \otimes \overline{F}\prns{Y}
\ar[r]^-{\sim}
&
\overline{F}\prns{X_{\mr{red}}\times_{\kk}Y}
}
\]
in which the vertical arrows are equivalences by $\mr{cdh}$-excision, since $X_{\mr{red}} \hookrightarrow X$ is a universal homeomorphism.
In general, suppose $\alpha_{SY}$ is an equivalence for each separated $\kk$-scheme $S$ of dimension $< \dim{X}$ and consider the square $Q$ of \eqref{rect.5.1}.
Tensoring $\overline{F}\prns{Q}$ with $\overline{F}\prns{Y}$ and again using the fact that $\mc{C}^{\otimes}$ is a stable symmetric monoidal qcategory, it suffices to show that $\alpha_{ZY}$, $\alpha_{\overline{Z}Y}$ and $\alpha_{\overline{X}Y}$ are equivalences.
However, $\alpha_{ZY}$ and $\alpha_{\overline{Z}Y}$ are equivalences by the inductive hypothesis, and $\alpha_{\overline{X}Y}$ is also an equivalence: permuting the tensor factors, it becomes $\alpha_{Y\overline{X}}$ and $\overline{X}$ is smooth, so we are in the previous case.

If $Y$ is separated and $X$ is arbitrary, choose a finite Zariski cover $\brc{j_{\beta}: X_{\beta} \hookrightarrow X}_{1 \leq \beta \leq n}$ such that $X_{\beta}$ is separated for each $1 \leq \beta \leq n$ and $n > 1$.
Let $X' := \bigcup_{1 \leq \beta <n}X_{\beta}$.
We have an element
\[
\xymatrix{
X' \cap X_n
\ar[r]_-{j''}
\ar[d]_-{j'_n}
\ar@{}[dr]|Q
&
X_n
\ar[d]^-{j_n}
\\
X'
\ar[r]^-{j'}
&
X
}
\]
of $\mf{Q}_{\mr{Nis}}\prns{\sch[ft]{\kk}}$ and $X' \cap X_n$ and $X'$ are both unions of $n-1$ separated open subschemes.
Applying the Nisnevich-excisive functor $\overline{F}$ to the square $Q$ and tensoring $\overline{F}\prns{Q}$ with $\overline{F}\prns{Y}$, we find by induction on $n$ that $\alpha_{XY}$ is an equivalence.
Inducting now on the number of elements in a Zariski cover of $Y$ by separated $S$-subschemes, we find that $\alpha_{XY}$ is an equivalence for arbitrary $X$ and $Y$.
\end{proof}

\begin{thm}
\label{rect.6}
\
\begin{enumerate}
\item
There exists a symmetric monoidal functor $\Gamma^{\otimes}_{\mr{Hdg}}: \prns{\sch[ft]{\kk}}\scr{\mr{op}, \amalg}{} \to \D{\ind{\mhs^{\mr{p}}_{\KK}}}\scr{\otimes}{}$ such that, for each $X \in \sch[ft]{\kk}$ and each $r \in \mb{Z}$, $\h^r\Gamma_{\mr{Hdg}}\prns{X}$ is naturally isomorphic to $\mr{H}^r_{\mr{Betti}}\prns{X \otimes_{\kk}\mb{C},\KK}$ equipped with the mixed Hodge $\KK$-structure of \textup{\cite[8.2.1]{Deligne_hodgeIII}}.
\item
The underlying functor $\Gamma_{\mr{Hdg}}$ factors up to equivalence as 
\[
\prns{\sch[ft]{\kk}}\op 
\xrightarrow{\tilde{\Gamma}_{\mr{Hdg}}} 
\calg{\cpx{\ind{\mhs^{\mr{p}}_{\KK}}}\scr{\otimes}{}}
\xrightarrow{\phi'}
\D{\ind{\mhs^{\mr{p}}_{\KK}}},
\]
where $\phi'$ is induced by the functor $\phi$ of \textup{\ref{tannakian.10}}.
\end{enumerate}
\end{thm}

\begin{proof}
For brevity, let $\mc{C}^{\otimes} := \D{\ind{\mhs^{\mr{p}}_{\KK}}}\scr{\otimes}{}$.
Let $F^{\otimes}: \prns{\sm[sft]{\kk}}\scr{\mr{op}, \amalg}{} \to \mc{C}^{\otimes}$ denote the symmetric monoidal functor denoted by $\Gamma^{\otimes}_{\mr{Hdg}}$ in \ref{rect.3}.
Composing $F$ with the ``underlying complex of $\KK$-modules'' functor $\omega: \mc{C} \to \D{\mod{\ab}_{\KK}}$ results in a functor assigning to each $X \in \sm[sft]{\kk}$ the $\KK$-linear singular cochain complex of $X\prns{\mb{C}}^{\mr{an}}$ or, in any case, a complex of $\KK$-modules quasi-isomorphic to it.

Betti cohomology is $\mb{A}^1$-homotopy invariant and Nisnevich excisive:
the $\mb{A}^1$-invariance follows from the contractibility of $\mathbf{A}^1_{\mb{C}}\prns{\mb{C}}^{\mr{an}}$ and homotopy invariance of singular cohomology;
one may check that it is Nisnevich excisive by standard cohomological descent arguments (\cite[Expos\'e V \emph{bis}, 4.1.8, 5.2.3]{SGA4b}), or use the fact (\cite[3.3]{Ayoub_operations-de-Grothendieck}) that the derived categories $\mr{D}\prns{X\prns{\mb{C}}^{\mr{an}},\ \KK}$ of analytic sheaves of complexes of $\KK$-modules form a stable homotopy 2-functor (\cite[1.4.1]{Ayoub_six-operationsI}) and remark that the existence of localization sequences (\cite[1.4.9]{Ayoub_six-operationsI}) implies Nisnevich excision.
As $\omega$ is conservative by \ref{tannakian.6}, $F$ is also $\mb{A}^1$-invariant and Nisnevich excisive. 
Claim $(i)$ now follows from \ref{rect.5}.

By \cite[2.4.3.18]{Lurie_higher-algebra}, the functor $\Gamma^{\otimes}_{\mr{Hdg}}$ of $(i)$ is classified by an essentially unique functor 
$
\tilde{\Gamma}'_{\mr{Hdg}}:
\prns{\sch[ft]{\kk}}\op
\to
\calg{\mc{C}^{\otimes}}
$.
By \ref{tannakian.10}, we therefore have a functor 
\[
\phi^{-1}\tilde{\Gamma}'_{\mr{Hdg}}:
\prns{\sch[ft]{\kk}}\op 
\to
\calg{\cpx{\ind{\mhs^{\mr{p}}_{\KK}}}\scr{\otimes}{}}\brk{\mf{W}^{-1}}.
\]
By \ref{tannakian.8} and \cite[1.3.4.25]{Lurie_higher-algebra}, we can rectify $\phi^{-1}\tilde{\Gamma}'_{\mr{Hdg}}$ to a functor $\tilde{\Gamma}_{\mr{Hdg}}: \prns{\sch[ft]{\kk}}\op \to \calg{\cpx{\ind{\mhs^{\mr{p}}_{\KK}}}\scr{\otimes}{}}$, which proves $(ii)$.
\end{proof}

\section{Integral coefficients}

\setcounter{thm}{-1}

\begin{notation}
\label{int.0}
Throughout this section, we fix the following:
\begin{enumerate}
\item 
$\bs{\Lambda} \hookrightarrow \mb{R}$, a Noetherian subring of the real numbers of global dimension $\leq 1$ such that $\bs{\Lambda} \otimes_{\mb{Z}}\mb{Q}$ is a field, e.g., $\bs{\Lambda} \in \brc{\mb{Z}, \mb{Q}, \mb{R}}$;
\item
$\KK := \bs{\Lambda} \otimes_{\mb{Z}} \mb{Q}$; and
\item
$\kk \hookrightarrow \mb{C}$, a subfield of the complex numbers.
\end{enumerate}
\end{notation}

\begin{motivation}
Working in the $\KK$-linear setting has simplified things in several ways:
it allowed us to apply the results of \cite[\S2]{Drew_verdier-quotients} to pass from symmetric monoidal differential graded categories to symmetric monoidal qcategories without being forced to deal with cofibrant resolutions of our differential graded categories; 
it allowed us to equip the bounded derived category of $\mhs^{\mr{p}}_{\KK}$ with a symmetric monoidal structure without constructing flat resolutions; and
it allowed us to apply the rectification of \ref{tannakian.10}.

In this section, we show that it is possible to work with integral rather than rational coefficients.
However, whereas in \ref{rect.6} we constructed a presheaf of strictly commutative differential graded algebras, in the integral setting, one may at best hope for a presheaf of $\mb{E}_{\infty}$-algebras.
In fact, we content ourselves to ask for a presheaf of $\mb{E}_{\infty}$-algebras at the level of symmetric monoidal qcategories and set aside the question of establishing an analogue of the rectification result \ref{tannakian.10} for $\mb{E}_{\infty}$-algebras with integral coefficients.
\end{motivation}

\begin{summary}
We begin by constructing a t-structure on the fiber product of two stable qcategories equipped with t-structures over a third (\ref{int.2}) and studying the heart of this t-structure (\ref{int.3}).
We then apply this to show that the derived qcategory of $\mhs^{\mr{p}}_{\bs{\Lambda}}$ is the fiber product of $\D{\mhs^{\mr{p}}_{\KK}}^{\mr{b}}$ and $\D{\mod{\ab}_{\bs{\Lambda}}\scr{}{\aleph_0}}^{\mr{b}}$ over $\D{\mod{\ab}_{\KK}\scr{}{\aleph_0}}$ (\ref{int.5}).
After constructing a symmetric monoidal functor computing $\bs{\Lambda}$-linear Betti cohomology (\ref{int.8}), this allows us to establish in \ref{int.9} the $\bs{\Lambda}$-linear analogue of \ref{rect.6}$(i)$.
\end{summary}

\begin{lemma}
\label{int.1}
Consider a commutative diagram
\begin{equation}
\label{int.1.1}
\xymatrix{
\mc{B}
\ar[d]_-{b}
\ar[r]_-{f}
&
\mc{A}
\ar[d]^-{a}
&
\mc{C}
\ar[l]^-{g}
\ar[d]^-{c}
\\
\mc{B}'
\ar[r]^-{f'}
&
\mc{A}'
&
\mc{C}'
\ar[l]_-{g'}
}
\end{equation}
in $\qcat$.
If $a$, $b$ and $c$ are fully faithful, then so is the induced functor $\phi: \mc{B} \times_{\mc{A}} \mc{C} \to \mc{B}' \times_{\mc{A}'} \mc{C}'$.
\end{lemma}

\begin{proof}
By \cite[3.3.3.2]{Lurie_higher-topos}, an object of the fiber product $\mc{B} \times_{\mc{A}} \mc{C}$ in $\qcat$ is determined by objects $A \in \mc{A}$, $B \in \mc{B}$, $C \in \mc{C}$ and equivalences $fB \simeq A \simeq gC$.
Moreover, if $D$ and $D'$ are objects of $\mc{B} \times_{\mc{A}} \mc{C}$ corresponding to such $\prns{A, B, C}$ and $\prns{A', B', C'}$, respectively, then we have
\begin{equation}
\label{int.1.2}
\map{D}{D'}_{\mc{B} \times_{\mc{A}}\mc{C}}
\simeq
\map{B}{B'}_{\mc{B}} \times_{\map{A}{A'}_{\mc{A}}} \map{C}{C'}_{\mc{C}}.
\end{equation}
This also applies to $\mc{B}' \times_{\mc{A}'} \mc{C}'$.
If $a$, $b$ and $c$ are fully faithful, then the induced morphism
\[
\map{B}{B'}_{\mc{B}} \times_{\map{A}{A'}_{\mc{A}}} \map{C}{C'}_{\mc{C}}
\to
\map{bB}{bB'}_{\mc{B}'} \times_{\map{aA}{aA'}_{\mc{A}'}} \map{cC}{cC'}_{\mc{C}'}
\]
is the fiber product of three equivalences in $\spc{}$, hence an equivalence itself.
\end{proof}

\begin{prop}
\label{int.2}
If $f: \mc{B} \to \mc{A}$ and $g: \mc{C} \to \mc{A}$ are t-exact functors between $\mf{U}$-small stable qcategories equipped with t-structures, then $\mc{D} := \mc{B} \times_{\mc{A}} \mc{C}$ admits a t-structure with respect to which the canonical functors $g': \mc{D} \to \mc{B}$ and $f': \mc{D} \to \mc{C}$ are t-exact.
\end{prop}

\begin{proof}
By \cite[1.1.4.2]{Lurie_higher-algebra}, $f'$ and $g'$ are morphisms of $\qcat^{\mr{Ex}}$.
Define $\mc{D}^{\leq 0} \subseteq \mc{D}$ to be the full subqcategory spanned by the objects $D$ such that $g'D \in \mc{B}^{\leq0}$ and $f'D \in \mc{C}^{\leq0}$, and define $\mc{D}^{> 0} \subseteq \mc{D}$ analogously.
If these define a t-structure, then $f'$ and $g'$ are necessarily t-exact.
Since $f'$ and $g'$ are exact, the suspension functor $\Sigma$ preserves $\mc{D}^{\leq0}$ and the loop space functor $\Omega$ preserves $\mc{D}^{> 0}$.

We claim that, if $X \in \mc{D}^{\leq0}$ and $Y \in \mc{D}^{>0}$, then $\pi_0\map{X}{Y}_{\mc{D}} = 0$.
As in \eqref{int.1.2}, we have
\[
\map{X}{Y}_{\mc{D}} 
\simeq 
\map{g'X}{g'Y}_{\mc{B}} \times_{\map{fg'X}{fg'Y}_{\mc{A}}} \map{f'X}{f'Y}_{\mc{C}}.
\]
This homotopy fiber product in $\sset$ induces an exact sequence of homotopy groups 
\begin{equation}
\label{int.2.1}
\pi_1\map{fg'X}{fg'Y}_{\mc{A}}
\to
\pi_0\map{X}{Y}_{\mc{D}}
\to
\pi_0\map{g'X}{g'Y}_{\mc{B}} \oplus \pi_0\map{f'X}{f'Y}_{\mc{C}}
\end{equation}
with base points given by the zero morphisms.
The last term vanishes since $g'X \in \mc{B}^{\leq 0}$, $g'Y \in \mc{B}^{>0}$, $f'X \in \mc{C}^{\leq0}$ and $f'Y \in \mc{C}^{>0}$; the first is isomorphic to
\[
\pi_0\Omega\map{fg'X}{fg'Y}_{\mc{A}}
\simeq
\pi_0\map{\Sigma fg'X}{fg'Y}_{\mc{A}}
=
0,
\]
since $\Sigma fg'X \in \mc{A}^{\leq 0}$ and $fg'Y \in \mc{A}^{>0}$.

Let $X \in \mc{D}$.
We claim that there exists a fiber sequence $X^{\leq0} \to X \to X^{>0}$ in $\mc{D}$ with $X^{\leq0} \in \mc{D}^{\leq0}$ and $X^{>0} \in \mc{D}^{>0}$.
Consider the functors $\delta: \Delta^1 \to \mc{B}$ and $\delta': \Delta^1 \to \mc{C}$ corresponding to the canonical morphisms $\eta: g'X \to \trun^{>0}g'X$ and $\eta': f'X \to \trun^{>0}f'X$, respectively.
We have a homotopy $f\delta \simeq g\delta'$ since $f$ and $g$ are t-exact.
By the universal property of the fiber product $\mc{D}$, $\delta$ and $\delta'$ induce a functor $\Delta^1 \to \mc{D}$ corresponding to a morphism $\tilde{\eta}: X \to X^{>0}$.
By construction, $X^{>0} \in \mc{D}^{>0}$.
Since $f'$ and $g'$ are exact and send $\tilde{\eta}$ to $\eta'$ and $\eta$, respectively, they send the fiber $X^{\leq0}$ of $\tilde{\eta}$ to $\trun^{\leq0}f'X$ and $\trun^{\leq0}g'X$, respectively, so $X^{\leq0} \in \mc{D}^{\leq0}$, as required.
\end{proof}

\begin{prop}
\label{int.3}
Consider a commutative diagram \textup{\eqref{int.1.1}} in which $f'$ and $g'$ are t-exact functors between $\mf{U}$-small stable qcategories equipped with t-structures and $a$, $b$ and $c$ are the fully faithful inclusions of the hearts of these t-structures.
If $f$ is an isofibration, then the heart of the t-structure of \textup{\ref{int.2}} on $\mc{B}' \times_{\mc{A}'} \mc{C}'$ is equivalent to the fiber product $\mc{B} \times_{\mc{A}} \mc{C}$, formed in the $1$-category $\cat$ of $\mf{U}$-small categories.
\end{prop}

\begin{proof}
Recall that an \emph{isofibration} is a functor of categories $F: \mc{D} \to \mc{D}'$ such that, for each $D \in \mc{D}$ and each isomorphism $\alpha: FD \isom D'$ in $\mc{D}'$, there exists an isomorphism $\beta: D \isom \tilde{D}'$ such that $F\beta = \alpha$.
By \cite[2.8]{Joyal_notes-on-quasi-categories}, the nerve functor $\nerve{}: \cat \to \sset$ is right Quillen with respect to the model structure on $\cat$ whose weak equivalences and fibrations are the equivalences and the isofibrations, respectively, and the Joyal model structure on $\sset$ (\cite[2.2.5.1]{Lurie_higher-topos}).
Each object is fibrant in this model structure on $\cat$, so the model structure is right proper.
In particular, pullbacks along isofibrations are homotopy pullbacks (\cite[1.19]{Barwick_left-and-right}).
Right Quillen functors between $\mf{U}$-combinatorial model categories induce $\mf{U}$-continuous functors between the underlying qcategories (\cite[1.3.4.26]{Lurie_higher-algebra}), so $\nerve{}$ sends the fiber product of $\mc{B}$ and $\mc{C}$ over $\mc{A}$ in $\cat$ to their fiber product in $\qcat$.
By \ref{int.1}, it follows that the induced functor $\mc{B} \times_{\mc{A}} \mc{C} \to \mc{B}' \times_{\mc{A}'} \mc{C}'$ is fully faithful.
Its essential image is contained in $\prns{\mc{B}' \times_{\mc{A}'} \mc{C}'}^{\heartsuit}$, so it remains to show that the essential image contains each $X \in \prns{\mc{B}' \times_{\mc{A}'} \mc{C}'}^{\heartsuit}$.
Such an object $X$ is classified by a functor $\chi: \Delta^0 \to \mc{B}' \times_{\mc{A}'} \mc{C}'$, which is in turn classified by functors $\chi_0: \Delta^0 \to \mc{B}'$ and $\chi_1: \Delta^0 \to \mc{C}'$ and a homotopy $f'\chi_0 \simeq g'\chi_1$. 
By construction of the t-structure on $\mc{B}' \times_{\mc{A}'} \mc{C}'$, $\chi_0$ and $\chi_1$ must factor through $\mc{B} = \mc{B}'^{\heartsuit}$ and $\mc{C} = \mc{C}'^{\heartsuit}$, respectively, so $\chi$ factors through $\mc{B} \times_{\mc{A}} \mc{C}$, as desired.
\end{proof}

\begin{lemma}
\label{int.4}
The natural functor $\mhs^{\mr{p}}_{\bs{\Lambda}} \to \mhs^{\mr{p}}_{\KK} \times_{\mod{\ab}_{\KK}\scr{}{\aleph_0}} \mod{\ab}_{\bs{\Lambda}}\scr{}{\aleph_0}$ in $\qcat$ is an equivalence.
\end{lemma}

\begin{proof}
Following the discussion in the proof of \ref{int.3}, it suffices to show that we have the corresponding equivalence in $\cat$ and that the fiber functor $\omega: \mhs^{\mr{p}}_{\KK} \to \mod{\ab}_{\KK}\scr{}{\aleph_0}$ sending polarizable mixed Hodge $\KK$-structure to its underlying $\KK$-module is an isofibration.
The former is true by definition of $\mhs^{\mr{p}}_{\bs{\Lambda}}$: its objects are given by pairs $\prns{M, V}$ such that $M$ is a $\bs{\Lambda}$-module of finite type and $V$ is an object of $\mhs^{\mr{p}}_{\KK}$ whose underlying $\KK$-module is $M \otimes_{\bs{\Lambda}}\KK$, and similarly for morphisms.
Also, we can push polarizable mixed Hodge $\KK$-structures forward along isomorphisms of $\KK$-modules to isomorphic mixed Hodge $\KK$-structures, so $\omega$ is an isofibration.
\end{proof}

\begin{thm}
\label{int.5}
There are natural equivalences
\begin{enumerate}
\item
$\phi: \D{\mhs^{\mr{p}}_{\bs{\Lambda}}}^{\mr{b}} \isom \tD{\mhs^{\mr{p}}_{\bs{\Lambda}}}^{\mr{b}} :=  \D{\mhs^{\mr{p}}_{\KK}}^{\mr{b}} \times_{\D{\mod{\ab}_{\KK}\scr{}{\aleph_0}}^{\mr{b}}} \D{\mod{\ab}_{\bs{\Lambda}}\scr{}{\aleph_0}}^{\mr{b}}$,
\item
$\Phi: \D{\ind{\mhs^{\mr{p}}_{\bs{\Lambda}}}} \isom \D{\ind{\mhs^{\mr{p}}_{\KK}}} \times_{\D{\mod{\ab}_{\KK}}} \D{\mod{\ab}_{\bs{\Lambda}}}$.
\end{enumerate}
\end{thm}

\begin{proof}
The functors in question are those induced by the universal properties of the fiber products defining the codomains.
They are thus exact by \cite[1.1.4.2]{Lurie_higher-algebra}. 
The functor $\phi$ is t-exact with respect to the natural t-structure on its domain and the one on its codomain induced by \ref{int.2} and the natural t-structures on $\D{\mhs^{\mr{p}}_{\KK}}^{\mr{b}}$, $\D{\mod{\ab}_{\KK}\scr{}{\aleph_0}}^{\mr{b}}$ and $\D{\mod{\ab}_{\bs{\Lambda}}\scr{}{\aleph_0}}^{\mr{b}}$.
By \ref{int.3} and \ref{int.4}, $\phi$ induces an equivalence between the hearts of these t-structures.
As $\bs{\Lambda}$ is of global dimension $\leq 1$ (\ref{int.0}), $\on{hdim}\prns{\mod{\ab}_{\bs{\Lambda}}\scr{}{\aleph_0}} = 1$ (\ref{tannakian.5}$(ii)$); $\on{hdim}\prns{\mhs^{\mr{p}}_{\KK}} = 1$ by \cite[3.35]{Peters-Steenbrink_mixed-hodge}; and $\on{hdim}\prns{\mod{\ab}_{\KK}\scr{}{\aleph_0}} = 0$.
Using the exact sequence of \eqref{int.2.1}, these homological-dimension computations imply that 
\[
\pi_0\map{M}{N \brk{r}}_{\tD{\mhs^{\mr{p}}_{\bs{\Lambda}}}^{\mr{b}}} 
= 
0
\]
for each $r \in \mb{Z}_{\geq 2}$ and each pair of objects $M$ and $N$ of $\tD{\mhs^{\mr{p}}_{\bs{\Lambda}}}^{\mr{b}}^{\heartsuit}$.
By the argument of \cite[2]{Wildeshaus_erratum}, which develops a remark of \cite[p.3]{Deligne-Goncharov_groupes-fondamentaux} on the proof of \cite[4.2]{Buchsbaum_satellites-and-universal}, $\phi$ is fully faithful.
Since the t-structures on $\D{\mhs^{\mr{p}}_{\KK}}^{\mr{b}}$, $\D{\mod{\ab}_{\bs{\Lambda}}\scr{}{\aleph_0}}^{\mr{b}}$ and $\D{\mod{\ab}_{\KK}\scr{}{\aleph_0}}^{\mr{b}}$, and hence also $\tD{\mhs^{\mr{p}}_{\bs{\Lambda}}}^{\mr{b}}$, are all bounded, the smallest stable subqcategory of $\tD{\mhs^{\mr{p}}_{\bs{\Lambda}}}^{\mr{b}}$ containing the heart is $\tD{\mhs^{\mr{p}}_{\bs{\Lambda}}}^{\mr{b}}$ itself, and it follows that $\phi$ is essentially surjective.

By \cite[5.3.2.11(3)]{Lurie_higher-algebra}, the canonical functor 
\[
\xymatrix{
\ind[1]{\D{\mhs^{\mr{p}}_{\KK}}^{\mr{b}} \times_{\D{\mod{\ab}_{\KK}\scr{}{\aleph_0}}^{\mr{b}}} \D{\mod{\ab}_{\bs{\Lambda}}\scr{}{\aleph_0}}^{\mr{b}}}
\ar[d]^{\psi}
\\
\ind{\D{\mhs^{\mr{p}}_{\KK}}^{\mr{b}}} \times_{\ind{\D{\mod{\ab}_{\KK}\scr{}{\aleph_0}}^{\mr{b}}}} \ind{\D{\mod{\ab}_{\bs{\Lambda}}\scr{}{\aleph_0}}^{\mr{b}}}
}
\]
is an equivalence.
By \cite[4.6]{Drew_verdier-quotients}, there is a canonical equivalence $\ind{\D{\mb{A}}^{\mr{b}}} \simeq \D{\ind{\mb{A}}}$ for each $\mf{U}$-small Noetherian Abelian category $\mb{A}$ such that $\on{hdim}\prns{\mb{A}} < \infty$.
Under these canonical equivalences, $\Phi$ is equivalent to $\psi \ind{\phi}$.
\end{proof}

\begin{lemma}
\label{int.6}
If $\mc{C}^{\otimes}$ is a $\mf{U}$-small stable symmetric monoidal qcategory such that each $V \in \mc{C}$ is $\otimes$-dualizable, then there is a symmetric monoidal equivalence $\mc{C}^{\mr{op},\otimes} \isom \mc{C}^{\otimes}$, where $\mc{C}^{\mr{op},\otimes}$ is the symmetric monoidal structure of \textup{\cite[2.4.2.7]{Lurie_higher-algebra}}.
\end{lemma}

\begin{proof}
Let $\mc{C}^{\otimes^{\mr{rev}}}$ denote the ``reverse'' symmetric monoidal structure on $\mc{C}$ given informally by $V \otimes^{\mr{rev}} W := W \otimes V$.
Specifically, we define $\mc{C}^{\otimes^{\mr{rev}}}$ as the pullback of the coCartesian fibration $\mc{C}^{\otimes} \to \mb{E}^{\otimes}_{\infty}$ along the reversal involution $\mb{E}^{\otimes}_{\infty} \to \mb{E}^{\otimes}_{\infty}$ of \cite[5.2.5.25]{Lurie_higher-algebra}. 
This reversal involution is homotopic to the identity, since $\mb{E}^{\otimes}_{\infty} \simeq \fin$ is the final object of the qcategory of \inftyone-operads, from which we deduce that $\mc{C}^{\otimes^{\mr{rev}}}$ is equivalent to $\mc{C}^{\otimes}$.

Applying \cite[5.2.5.27]{Lurie_higher-algebra} in the case $k = \infty$, i.e., taking the limit of the functors
$\alg{\qcat^{\times}}_{\mb{E}_k} \to \alg{\on{\mc{CP}air}}_{\mb{E}_k}$ constructed therein as $k$ approaches $\infty$, we obtain a pairing of symmetric monoidal qcategories (\cite[5.2.2.21]{Lurie_higher-algebra}) $\on{\mc{D}ual}\prns{\mc{C}}^{\otimes} \to \mc{C}^{\otimes} \times_{\fin} \mc{C}^{\otimes^{\mr{rev}}}$.
By \cite[5.2.2.25]{Lurie_higher-algebra}, this pairing induces a morphism of \inftyone-operads $\prns{-}^{\vee,\otimes}: \mc{C}^{\mr{op},\otimes} \to \mc{C}^{\otimes^{\mr{rev}}}$, which we refer to as the duality functor.
The functor $\prns{-}^{\vee}: \mc{C}\op \to \mc{C}$ underlying $\prns{-}^{\vee,\otimes}$ is given informally by $V \mapsto V^{\vee} := \umor{V}{\1{\mc{C}}}_{\mc{C}}$, where $\umor{-}{-}_{\mc{C}}$ is the internal $\hom{}{}$-object bifunctor.
Since each object of $\mc{C}$ is $\otimes$-dualizable by hypothesis, $\prns{-}^{\vee,\otimes}$ is in fact a symmetric monoidal functor:
the natural morphism $V^{\vee} \otimes^{\mr{rev}} W^{\vee} \to \prns{W \otimes V}^{\vee}$ is an equivalence. 

The functor underlying $\prns{-}^{\vee,\otimes}$ is an equivalence.
Indeed, essential surjectivity follows from the observation that $V \simeq \prns{V^{\vee}}\scr{\vee}{}$ for each $V \in \mc{C}$ and full faithfulness from the fact that 
\[
\map{V}{W}_{\mc{C}\op}
\to
\map{V^{\vee}}{W^{\vee}}_{\mc{C}}
\simeq
\map{W}{V}_{\mc{C}}
\]
is an equivalence for each $\prns{V,W} \in \mc{C}^2$.
By \cite[2.1.3.8]{Lurie_higher-algebra}, as the underlying functor $\prns{-}^{\vee}$ is an equivalence, so is $\prns{-}^{\vee,\otimes}$.
Composing $\prns{-}^{\vee,\otimes}$ with the above equivalence $\mc{C}^{\otimes^{\mr{rev}}} \simeq \mc{C}^{\otimes}$ completes the proof.
\end{proof}

\begin{thm}
[{\cite[\S1]{Ayoub_operations-de-Grothendieck}}]
\label{int.7}
There is a symmetric monoidal functor $\Gamma^{\otimes}_{\mr{Betti}}: \prns{\sm[ft]{\kk}}\scr{\mr{op}, \amalg}{} \to \D{\mod{\ab}_{\bs{\Lambda}}}\scr{\otimes}{}$ such that the induced functor $\tilde{\Gamma}_{\mr{Betti}}: \prns{\sm[ft]{\kk}}\op \to \calg{\D{\mod{\ab}_{\bs{\Lambda}}}\scr{\otimes}{}}$ \textup{(\cite[2.4.3.18]{Lurie_higher-algebra})} sends each $X \in \sm[ft]{\kk}$ to an $\mb{E}_{\infty}$-algebra $\tilde{\Gamma}_{\mr{Betti}}\prns{X}$ naturally equivalent to the $\bs{\Lambda}$-linear singular cochain complex of $X\prns{\mb{C}}^{\mr{an}}$ equipped with the $\mb{E}_{\infty}$-algebra structure corresponding to the cup product.
\end{thm}

\begin{proof}
Here, $\D{\mod{\ab}_{\bs{\Lambda}}}\scr{\otimes}{}$ denotes the stable locally $\mf{U}$-presentable symmetric monoidal qcategory underlying the projective model structure on $\cpx{\mod{\ab}_{\bs{\Lambda}}}\scr{\otimes}{}$ (\cite[4.2.13]{Hovey_model-categories}).
Let $\DA{\kk, \Lambda}\scr{\otimes}{}$ denote the stable locally $\mf{U}$-presentable symmetric monoidal qcategory underlying the $\mf{U}$-combinatorial symmetric monoidal model category used to define $\mb{DA}\prns{S, \bs{\Lambda}}$ in \cite[3.3]{Ayoub_realisation-etale} for $S = \spec{\kk}$.
This is a $\bs{\Lambda}$-linear variant of the construction of the $\mb{P}^1$-stable $\mb{A}^1$-homotopy category $\SH{\kk}$ of \cite[5.7]{Voevodsky_A1-homotopy-theory}.
This category is equipped with a canonical symmetric monoidal functor $\Sigma^{\infty}_{\mb{T}}: \prns{\sm[ft]{\kk}}^{\times} \to \DA{\kk, \bs{\Lambda}}\scr{\otimes}{}$ whose essential image is spanned by $\aleph_0$-presentable objects (\cite[1.3]{Riou_dualite-de-spanier-whitehead}).
From the symmetric monoidal Quillen equivalences constructed in \cite[\S1]{Ayoub_operations-de-Grothendieck}, we deduce a $\mf{U}$-cocontinuous symmetric monoidal functor $\varrho^{*, \otimes}: \DA{\kk, \bs{\Lambda}}\scr{\otimes}{} \to \D{\mod{\ab}_{\bs{\Lambda}}}\scr{\otimes}{}$. 
By the results of \cite[\S3]{Ayoub_operations-de-Grothendieck}, $\varrho^*$ preserves $\aleph_0$-presentable objects and, at the level of underlying homotopy categories, $\varrho^*$ is compatible with Grothendieck's six operations.
By \cite[1.4, 2.2]{Riou_dualite-de-spanier-whitehead}, since $\kk$ admits resolution of singularities (\cite{Hironaka_resolution-of-singularities}), the full subqcategory $\DA{\kk, \bs{\Lambda}}\scr{}{\aleph_0} \subseteq \DA{\kk, \bs{\Lambda}}$ is equal to the full subqcategory spanned by the $\otimes$-dualizable objects. 
Combining these results with \ref{int.6}, we obtain a symmetric monoidal functor $\Gamma^{\otimes}_{\mr{Betti}}$ given by the composite
\[
\prns{\sm[ft]{\kk}}\scr{\mr{op}, \amalg}{} 
\xrightarrow{\Sigma^{\infty,\mr{op}}_{\mb{T}}}
\prns{\DA{\kk, \bs{\Lambda}}\scr{}{\aleph_0}}\scr{\mr{op}, \otimes}{}
\isom
\DA{\kappa, \bs{\Lambda}}\scr{\otimes}{}
\xrightarrow{\varrho^{*, \otimes}}
\D{\mod{\ab}_{\bs{\Lambda}}}\scr{\otimes}{}.
\]
Unwinding the constructions and using the compatibility of $\varrho^*$ with Grothendieck's six operations, one finds that $\Gamma_{\mr{Betti}}\prns{X} \simeq f_*f^*\bs{\Lambda}_{X\prns{\mb{C}}^{\mr{an}}}$, where $f:X \to \spec{\kk}$ is the structural morphism and $\bs{\Lambda}_{X\prns{\mb{C}}^{\mr{an}}}$ is the constant sheaf associated to $\bs{\Lambda}$ on $X\prns{\mb{C}}^{\mr{an}}$.
Since $f_*f^*\bs{\Lambda}_{X\prns{\mb{C}}^{\mr{an}}}$ and the $\bs{\Lambda}$-linear singular cochain complex of $X\prns{\mb{C}}^{\mr{an}}$ are naturally equivalent as $\mb{E}_{\infty}$-algebras when the latter is equipped with the cup product, the claim follows.
\end{proof}

\begin{defn}
\label{int.8}
Let $\D{\ind{\mhs^{\mr{p}}_{\bs{\Lambda}}}}\scr{\otimes}{}$ denote the fiber product
\begin{equation}
\label{int.8.1}
\D{\ind{\mhs^{\mr{p}}_{\KK}}}\scr{\otimes}{} \times_{\D{\mod{\ab}_{\KK}}\scr{\otimes}{}} \D{\mod{\ab}_{\bs{\Lambda}}}\scr{\otimes}{}
\end{equation}
in $\calg{\pr^{\mr{L}, \otimes}}$, the qcategory of commutative algebras in the symmetric monoidal qcategory of locally $\mf{U}$-presentable qcategories and $\mf{U}$-cocontinuous functors \textup{(\cite[4.8.1.14]{Lurie_higher-algebra})}.
The functor $\mc{C}^{\otimes} \mapsto \calg{\mc{C}^{\otimes}}$ preserves fiber products, so, by \ref{int.5}, the qcategory underlying $\D{\ind{\mhs^{\mr{p}}_{\bs{\Lambda}}}}\scr{\otimes}{}$ is equivalent to $\D{\ind{\mhs^{\mr{p}}_{\bs{\Lambda}}}}$.
This circuitous construction of the symmetric monoidal structure on the derived category of $\ind{\mhs^{\mr{p}}_{\bs{\Lambda}}}$ facilitates the proof of \ref{int.9}.
Given any other more direct construction of this symmetric monoidal structure, it should not be difficult to show that it is equivalent to the above using the universal property of the fiber product \eqref{int.8.1}.
\end{defn}

\begin{cor}
\label{int.9}
There is a symmetric monoidal functor 
\[
\Gamma^{\otimes}_{\mr{Hdg}, \bs{\Lambda}}: 
\prns{\sch[ft]{\kk}}\scr{\mr{op}, \amalg}{} 
\to 
\D{\ind{\mhs^{\mr{p}}_{\bs{\Lambda}}}}\scr{\otimes}{}
\]
such that, for each $X \in \sch[ft]{\kk}$:
\begin{enumerate}
\item
$\h^r\Gamma_{\mr{Hdg}, \bs{\Lambda}}\prns{X}$ is naturally isomorphic to $\mr{H}^r_{\mr{Betti}}\prns{X, \bs{\Lambda}}$ equipped with the mixed Hodge $\bs{\Lambda}$-structure of \textup{\cite[8.2.1]{Deligne_hodgeIII}} for each $r \in \mb{Z}$; and
\item
when equipped with the $\mb{E}_{\infty}$-algebra structure induced by \textup{\cite[2.4.3.18]{Lurie_higher-algebra}}, the complex of $\bs{\Lambda}$-modules underlying $\Gamma_{\mr{Hdg},\bs{\Lambda}}\prns{X}$ is naturally equivalent to the $\bs{\Lambda}$-linear singular cochain complex of $X\prns{\mb{C}}^{\mr{an}}$ with the $\mb{E}_{\infty}$-algebra structure given by the cup product.
\end{enumerate}
\end{cor}

\begin{proof}
By \ref{rect.5} and the basic properties of Betti cohomology cited in the proof of \ref{rect.6}, the symmetric monoidal functor $\Gamma^{\otimes}_{\mr{Betti}}$ of \ref{int.7} extends to a symmetric monoidal functor $\prns{\sch[ft]{\kk}}\scr{\mr{op}, \amalg}{} \to \D{\mod{\ab}_{\bs{\Lambda}}}\scr{\otimes}{}$, abusively denoted by $\Gamma^{\otimes}_{\mr{Betti}}$, computing the $\bs{\Lambda}$-linear Betti cohomology of each $X \in \sch[ft]{\kk}$.
After composition with the evident symmetric monoidal functors $\D{\ind{\mhs^{\mr{p}}_{\KK}}}\scr{\otimes}{} \to \D{\mod{\ab}_{\KK}}\scr{\otimes}{}$ and $\D{\mod{\ab}_{\bs{\Lambda}}}\scr{\otimes}{} \to \D{\mod{\ab}_{\KK}}\scr{\otimes}{}$, respectively, the symmetric monoidal functors $\Gamma^{\otimes}_{\mr{Hdg}}$ of \ref{rect.6} and $\Gamma^{\otimes}_{\mr{Betti}}$ become equivalent: both compute $\KK$-linear Betti cohomology.
The claim now follows from the universal property of the fiber product \eqref{int.8.1}.
\end{proof}

\begin{rmk}
\label{int.10}
As mentioned in the introduction, from the symmetric monoidal functor $\Gamma^{\otimes}_{\mr{Hdg},\bs{\Lambda}}$, one may construct a motivic $\mb{E}_{\infty}$-ring spectrum representing $\bs{\Lambda}$-linear absolute Hodge cohomology (\cite[\S5]{Beilinson_absolute-hodge}).
This observation will be exploited and explored further in a forthcoming preprint. 
\end{rmk}

\allsectionsfont{\center\fontseries{m}\scshape}
\bibliographystyle{alpha}
\bibliography{all}
\end{document}